\documentclass[a4paper, 10pt, english, reqno, dvipsnames, table]{amsart}
\usepackage[utf8]{inputenc}   
\usepackage[T1]{fontenc}      

\usepackage{amsmath}
\usepackage{amsthm}
\usepackage{amssymb}
\usepackage{amsbsy}    
\usepackage{mathtools}  
\usepackage{mathrsfs}   
\usepackage{amsfonts}

\usepackage{tikz}
\usepackage{tikz-cd}
\usepackage{quiver}
\usepackage{lipsum}

\usepackage{geometry}
\theoremstyle{plain}
\newtheorem{theorem}{Theorem}[section]
\newtheorem{proposition}[theorem]{Proposition}
\newtheorem{lemma}[theorem]{Lemma}
\newtheorem{corollary}[theorem]{Corollary}

\newtheorem{definition-theorem}[theorem]{Definition-Theorem}
\newtheorem{definition-proposition}[theorem]{Definition-Proposition}

\theoremstyle{definition}

\theoremstyle{remark}
\newtheorem{remark}[theorem]{Remark}

\usepackage{paralist}   
\usepackage[figurewithin=none]{caption}
\usepackage{graphicx}
\usepackage{float}
\usepackage{xstring}
\usepackage{multirow}
\usepackage{array}
\usepackage{longtable}

\usepackage{xcolor}

\newcommand{\Gal}[2]{\operatorname{\mathbf{Gal}}(#1/#2)}
\newcommand{\ord}{\mathrm{ord}}

\newcommand{\im}{\mathrm{Im}}
\newcommand{\GL}{\mathrm{GL}}

\newcommand{\coker}{\mathrm{coker}}
\newcommand{\Aut}{\mathrm{Aut}}

\makeatletter
\def\l@subsection{\@tocline{2}{0pt}{2pc}{5pc}{}}
\makeatother

\usepackage{url}
\usepackage[unicode, colorlinks, linktocpage=true, backref=page]{hyperref}
\hypersetup{
	linkcolor = MidnightBlue,
	citecolor = ForestGreen,
	urlcolor = YellowOrange,
	bookmarksnumbered = true
}

\title[Torsion groups of $E/\mathbb{Q}$ over $\mathbb{Z}_p$-extensions of quadratic fields: the $p\le 5$ case]{Torsion groups of rational elliptic curves over $\mathbb{Z}_p$-extensions of quadratic fields: the $p\le 5$ case}
\author{Haidong Li}
\address[Haidong Li]{Jiangsu Police Institute. No.48 Shifo Sangong, Pukou District, Nanjing, Jiangsu, 210031, China.}
\email{lihaidong@jspi.cn}

\date{September 20, 2026.}
\subjclass[2020]{Primary 11G05, 11R23, 20K10}
\keywords{Elliptic Curves, $\mathbb{Z}_p$-extensions, Torsion Groups}

\begin{document}

\begin{abstract}
Let $E$ be a rational elliptic curve. We generalize a theorem due to Avc{\i}\cite{AVCI2026153}, which asserts that for any quadratic field $K$ and prime $p>5$, the equality $E(K)_{\mathrm{tors}} = E(L)_{\mathrm{tors}}$ holds for every $\mathbb{Z}_p$-extension $L/K$. In this paper, we consider the setting where the $\mathbb{Z}_p$-extension $L$ is replaced by the compositum $K_{\infty}$ of all $\mathbb{Z}_p$-extensions of $K$. Under this new setting, we prove the analogous statement for $p=5$, and further provide some results for the remaining primes $p=3$ and $p=2$, where we also try to study the growth pattern of torsion groups. Moreover, we completely classify all triples \((E,K,p)\) for which \(E(K_\infty)_{\mathrm{tors}}\) is infinite. 
\end{abstract}  
    
    \maketitle
    \tableofcontents

\section{Introduction}

Let \(E\) be an elliptic curve over a number field \(F\). By the Mordell--Weil theorem \cite{Mordell1922,Weil1928}, as an abelian group, \(E(F)\) is finitely generated. Consequently,
\[
E(F)\cong \mathbb{Z}^r\oplus E(F)_{\mathrm{tors}},
\]
where \(E(F)_{\mathrm{tors}}\) denotes the finite torsion subgroup of \(E(F)\).

It is a natural problem to classify all possible torsion subgroups \(E(F)_{\mathrm{tors}}\) for a number field \(F\). The first such result was obtained by Mazur \cite[Theorem (8)]{mazurModularCurvesEisenstein1977} for the base field \(\mathbb{Q}\), where he determined all possible torsion subgroups \(E(\mathbb{Q})_{\mathrm{tors}}\). The classification was later extended to quadratic fields by Kamienny \cite[Theorem 3.1]{kamiennyTorsionPointsElliptic1992} and Kenku--Momose \cite[Theorem (0.1)]{kenkumomose}, to cubic fields by Derickx--Etropolski--van Hoeij--Morrow--Zureick-Brown \cite[Theorem A]{DER2021}, and to quartic fields by Derickx--Najman \cite[Theorem 1.1]{derickx2025}. For general number fields of degree \(\ge 5\), the classification problem remains open.

A separate but equally natural direction is to consider the torsion subgroup \(E(L)_{\mathrm{tors}}\) for infinite algebraic extensions \(L/\mathbb{Q}\), such as \(\mathbb{Z}_p\)-extensions or their compositum. In contrast to the finite extension case, where \(E\) is defined over \(F\) and \(E(F)_{\mathrm{tors}}\) is the group of \(F\)-rational torsion points, we focus on rational elliptic curves \(E/\mathbb{Q}\) and \(E(L)_{\mathrm{tors}}\) for an algebraic extension \(L/\mathbb{Q}\). Chou \cite{chou2019} classified all possible torsion subgroups of \(E(\mathbb{Q}^{ab})\) for rational elliptic curves \(E/\mathbb{Q}\) and the maximal abelian extension of \(\mathbb{Q}\). By the Kronecker--Weber theorem \cite[Theorem 14.1]{Washington1997}, \(\mathbb{Q}^{\mathrm{ab}}\) is the compositum of all \(\mathbb{Q}(\mu_{p^\infty})\) as \(p\) ranges over all primes, where \(\mu_{p^\infty}\) denotes the group of all \(p\)-power roots of unity. Since each \(\mathbb{Q}(\mu_{p^\infty})\) contains the unique cyclotomic \(\mathbb{Z}_p\)-extension \(\mathbb{Q}_{\infty,p}\) of \(\mathbb{Q}\), it is natural to study torsion subgroups over \(\mathbb{Z}_p\)-extensions of more general number fields.

In this direction, Avc{\i} \cite{AVCI2026153} considered the case of \(\mathbb{Z}_p\)-extensions \(L\) over quadratic fields \(K\), and proved that for any rational elliptic curve \(E/\mathbb{Q}\),
\[
E(K)_{\mathrm{tors}} = E(L)_{\mathrm{tors}}
\]
whenever \(p>5\), for every quadratic field \(K\). In the present article, we generalize this result to \(p=5\), and also obtain some results for the primes \(p=3\) and \(p=2\).

\subsection{Main theorems}

\begin{theorem}\label{theorem:main}
    Let \(E/\mathbb{Q}\) be a rational elliptic curve and let \(K\) be a quadratic field. Denote by \(K_{\infty}\) the compositum of all \(\mathbb{Z}_p\)-extensions of \(K\), and let \(L/K\) be any single \(\mathbb{Z}_p\)-extension.
    \begin{enumerate}
        \item[(i)] For \(p\ge 5\), we have
        \[
        E(K_{\infty})_{\mathrm{tors}}=E(K)_{\mathrm{tors}}.
        \]
        \item[(ii)] For \(p=3\), the following hold:

        \begin{itemize}
    \item For any prime \(q>7\) or \(q=5\),
    \[
    E(K_{\infty})_{\mathrm{tors}}[q^{\infty}] = E(K)_{\mathrm{tors}}[q^{\infty}].
    \]
    \item For \(q=7\),
    \[
    E(K_\infty)_{\mathrm{tors}}[7^\infty]=E(K_{cyc}^1)_{\mathrm{tors}}[7^\infty].
    \]
    Here, \(K_{cyc}^1\) is the first layer of the cyclotomic \(\mathbb{Z}_3\)-extension.

    \item For \(q=2\),
    \[E(K_\infty)_\mathrm{tors}[2^\infty]=E(K_\infty^{(1)})_\mathrm{tors}[2^\infty]\]
    Here, \(K_\infty^{(1)}\) is the compositum of all first layers of \(\mathbb{Z}_3\)-extensions. Details about the \(2\)-primary part are described by Proposition~\ref{prop:2-torsion-KneqQsqrt-3}.

    \item For \(q=3\) and \(K\neq \mathbb{Q}(\sqrt{-3})\),
    \[
    E(K_\infty)_{\mathrm{tors}}[3^\infty]=E(K_{cyc})_{\mathrm{tors}}[3^\infty].
    \]
    Moreover, we have
    \begin{itemize}
         \item If \(E(K)_{\mathrm{tors}}[3^\infty]\) is trivial, then \(E(K_\infty)_{\mathrm{tors}}[3^\infty]\) is trivial.

         \item If \(E(K)_{\mathrm{tors}}[3^\infty] \cong \mathbb{Z}/9 \mathbb{Z}\), then \(E(K_\infty)_{\mathrm{tors}}[3^\infty] \cong \mathbb{Z}/9\mathbb{Z}\).

         \item If \(E(K)_{\mathrm{tors}}[3^\infty] \cong \mathbb{Z}/3 \mathbb{Z}\):
         
            \begin{itemize}
                \item If \(E\) or \(E^{(K)}\) is isomorphic to the elliptic curve with Cremona label \(\mathrm{27a4}\), then \(E(K_{cyc}^n)_{\mathrm{tors}}[3^\infty] \cong \mathbb{Z}/3^{n+1} \mathbb{Z}\) for \(n=1,2\) and \(E(K_{cyc}^n)_{\mathrm{tors}}[3^\infty] \cong \mathbb{Z}/27 \mathbb{Z}\) for \(n>2\). Hence in this case \(E(K_\infty)_\mathrm{tors}[3^\infty]\cong \mathbb{Z}/27\mathbb{Z}\).

                \item If  \(E\ncong \mathrm{27a4}\)  and  \(E^{(K)}\ncong \mathrm{27a4}\) , then either \(E(K_\infty)_\mathrm{tors}[3^\infty]=E(K)_{\mathrm{tors}}[3^\infty] \cong \mathbb{Z}/3 \mathbb{Z}\) or \(E(K_\infty)_\mathrm{tors}[3^\infty]\cong \mathbb{Z}/9 \mathbb{Z}\), with growth occurring in \(E(K_{cyc}^1)_\mathrm{tors}[3^\infty] \cong \mathbb{Z}/9\mathbb{Z}\). A criterion for whether the growth occurs, using Kummer theory and Galois cohomology, is given by Proposition~\ref{prop:3-kummertheory}.
            \end{itemize}
     \end{itemize}

    \item For \(q=3\) and \(K=\mathbb{Q}(\sqrt{-3})\), we have the following:
    \begin{itemize}
        \item \(E(K_\infty)_\mathrm{tors}[3^\infty]\) is infinite if and only if \(E\) is isomorphic to one of eight curves: \(\mathrm{27a1}\), \(\mathrm{27a2}\), \(\mathrm{27a3}\), \(\mathrm{27a4}\), \(\mathrm{243a1}\), \(\mathrm{243a2}\), \(\mathrm{243b1}\), \(\mathrm{243b2}\) with Cremona label, which are exactly the curves listed in \cite[Table~1]{xu2026torsiongrowthrationalelliptic}. In this case, \(E[3^\infty]\subseteq E(K_\infty)\) and \(E\) has CM by an order in \(K\), with good reduction outside \(3\).
        \item For every \(\mathbb{Z}_3\)-extension \(L/K\), \(E(L)_{\mathrm{tors}}[3^\infty]\) is finite. The growth pattern is given in Theorem~\ref{theorem:growth-patternQ(-3),-3}.
    \end{itemize}
           
\end{itemize}
        \item[(iii)] For \(p=2\), the following hold:
        
        \begin{itemize}
            \item For any prime \(q>7\),
            \[
            E(K_{\infty})_{\mathrm{tors}}[q^{\infty}] = E(K)_{\mathrm{tors}}[q^{\infty}]=0.
            \]
            \item For \(q=7\),
            \[
            E(K_\infty)_{\mathrm{tors}}[7^{\infty}] = E(K_\infty^{(1)})_{\mathrm{tors}}[7^{\infty}].
            \]
            \item For \(q=5\), \(E(K_\infty)_{\mathrm{tors}}[5^{\infty}]\) is at most a cyclic group of order \(5\) and
            \[
            E(K_\infty)_{\mathrm{tors}}[5^{\infty}] = E(K_\infty^{(2)})_{\mathrm{tors}}[5^{\infty}].
            \]
            \item For \(q=3\),
            \[
            E(K_\infty)_{\mathrm{tors}}[3^{\infty}] = E(K_\infty^{(3)})_{\mathrm{tors}}[3^{\infty}].
            \]
            The growth pattern is given in Proposition~\ref{prop:p=2,q=3}.
            \item For \(q=2\), we have the following.
                \begin{itemize}
                    \item \(E(K_\infty)_\mathrm{tors}[2^\infty]\) is infinite if and only if
                        \begin{itemize}
                            \item \(K=\mathbb{Q}(i)\) and \(E\) is isomorphic to one of the twelve curves: \(\mathrm{32a1}\), \(\mathrm{32a2}\), \(\mathrm{64a1}\), \(\mathrm{64a4}\), \(\mathrm{256b1}\), \(\mathrm{256b2}\), \(\mathrm{256c1}\), \(\mathrm{256c2}\), \(\mathrm{32a3}\), \(\mathrm{32a4}\), \(\mathrm{64a2}\), \(\mathrm{64a3}\), with Cremona label.
                            \item \(K=\mathbb{Q}(\sqrt{-2})\) and \(E\) is isomorphic to one of the four curves: \(\mathrm{256a1}\), \(\mathrm{256a2}\), \(\mathrm{256d1}\), \(\mathrm{256d2}\), with Cremona label.
                            
                            These \(16\) curves are listed in \cite[Table~3]{xu2026torsiongrowthrationalelliptic}. In this case, \(E[2^\infty]\subseteq E(K_\infty)\) and \(E\) has CM by an order in \(K\) with good reduction outside \(2\).
                        \end{itemize}
                    \item For every \(\mathbb{Z}_2\)-extension \(L/K\), \(E(L)_{\mathrm{tors}}[2^\infty]\) is finite. The growth pattern is given in Theorem~\ref{theorem:growth-pattern,2}.
                \end{itemize}

        \end{itemize}
        
    \end{enumerate}
\end{theorem}

The three parts of Theorem~\ref{theorem:main} are obtained as follows. Part~(i) is Theorem~\ref{theorem:p>=5}. For Part~(ii), the prime-to-\(3\) assertions follow from Theorems~\ref{theorem:p=3not-3} and~\ref{theorem:p=3=3}, together with Propositions~\ref{prop:p=3,q=7} and~\ref{prop:2-torsion-KneqQsqrt-3}. The \(3\)-primary assertions for \(K\neq\mathbb{Q}(\sqrt{-3})\) and \(K=\mathbb{Q}(\sqrt{-3})\) follow from Theorems~\ref{theorem:growthppterQneqQ-3}, \ref{theorem:p=3singleZpfinite}, \ref{theorem:CMclass,p=3}, and~\ref{theorem:growth-patternQ(-3),-3}. Part~(iii) follows from Theorem~\ref{theorem:p=2}, Propositions~\ref{prop:p=2,q=7}, \ref{prop:p=2,q=5}, and~\ref{prop:p=2,q=3}, and the \(2\)-primary results in Proposition~\ref{prop:p=2-infinite-dichotomy} and Theorems~\ref{theorem:CMclass,p=2}, \ref{theorem:p=2singleZpfinite}, and~\ref{theorem:growth-pattern,2}.

As a direct corollary, we have the following theorem.

\begin{theorem}\label{theorem:classifyfinite}
    Let \(E/\mathbb{Q}\) be an elliptic curve and \(K\) be a quadratic field with \(K_\infty\) the compositum of all \(\mathbb{Z}_p\)-extensions. Then \(E(K_\infty)_{\mathrm{tors}}\) is infinite if and only if the triple \((E,K,p)\) satisfies the following:
    \begin{itemize}
        \item \((K,p)=(\mathbb{Q}(\sqrt{-3}),3)\) and  \(E\) isomorphic to one of the following: \(\mathrm{27a1}\), \(\mathrm{27a2}\), \(\mathrm{27a3}\), \(\mathrm{27a4}\), \(\mathrm{243a1}\), \(\mathrm{243a2}\), \(\mathrm{243b1}\), \(\mathrm{243b2}\) with Cremona label.

        \item \((K,p)=(\mathbb{Q}(i),2)\) and  \(E\) isomorphic to one of the following: \(\mathrm{32a1}\), \(\mathrm{32a2}\), \(\mathrm{64a1}\), \(\mathrm{64a4}\), \(\mathrm{256b1}\), \(\mathrm{256b2}\), \(\mathrm{256c1}\), \(\mathrm{256c2}\), \(\mathrm{32a3}\), \(\mathrm{32a4}\), \(\mathrm{64a2}\), \(\mathrm{64a3}\), with Cremona label.

        \item \((K,p)=(\mathbb{Q}(\sqrt{-2}),2)\) and  \(E\) isomorphic to one of the following: \(\mathrm{256a1}\), \(\mathrm{256a2}\), \(\mathrm{256d1}\), \(\mathrm{256d2}\), with Cremona label.
    \end{itemize}
\end{theorem}

\begin{proof}

    By Theorem~\ref{theorem:main}(i), if \(p \ge 5\), then \(E(K_\infty)_\mathrm{tors}=E(K)_\mathrm{tors}\), where the latter one is finite by Mordell--Weil for all \(E\) and \(K\).

    Note that \[E(K_\infty)_\mathrm{tors}=\bigoplus_{q \text{ prime }} E(K_\infty)_\mathrm{tors}[q^\infty]. \tag{\(\ast\)}\]

    It suffices to show that \(E(K_\infty)_\mathrm{tors}[q^\infty]\) is trivial for all but finitely many \(q\) and  is finite for other remaining \(q\).

    If \(p=3, K\neq \mathbb{Q}(\sqrt{-3})\) and \(q \ge 11 \), by  Theorem~\ref{theorem:main}(ii), \(E(K_\infty)_\mathrm{tors}[q^\infty]=E(K)_\mathrm{tors}[q^\infty]\). Then by Theorem~\ref{theorem:NajmanK}, the latter is trivial. For the remaining primes \(q=2,3,5,7\),  \(E(K_\infty)[q^\infty]\) is finite since it equals a \(q\)-primary part of torsion subgroups over number fields by Theorem~\ref{theorem:main}(ii). Therefore, by \((\ast)\),  \(E(K_\infty)_\mathrm{tors}\) is finite in this case.

    If \(p=3, K=\mathbb{Q}(\sqrt{-3})\), the argument is exactly the same as previous paragraph when \(q\neq 3\). Then \(E(K_\infty)_\mathrm{tors}\) is finite if and only if \(E(K_\infty)[3^\infty]\) is finite. By the last bullet of Theorem~\ref{theorem:main}(ii), we are done.

    If \(p=2\), the argument is similar to \(p=3\). Now if \(K \neq \mathbb{Q}(i)\) or \(\mathbb{Q}(\sqrt{-2})\) then \(q\)-primary part is trivial for \(q>7\) and finite for \(q=2,3,5,7\) by Theorem~\ref{theorem:main}(iii). And if \(K=\mathbb{Q}(i)\) or \(\mathbb{Q}(\sqrt{-2})\), only \(2\)-primary part matters, which is  given by the last bullet of  Theorem~\ref{theorem:main}(iii).
\end{proof}

\begin{remark}
This work and the work of Xu--Yuan~\cite{xu2026torsiongrowthrationalelliptic} were carried out independently and simultaneously. There is some overlap between the two papers, but their emphases are different. Xu--Yuan study the torsion groups on the different \(\mathbb{Z}_p\)-slopes for several fixed curves, whereas the present article focuses on the layer-by-layer growth, including the first layer where growth can occur and the layer where it stabilizes.  

In the CM cases, the eight curves in \cite[Table~1]{xu2026torsiongrowthrationalelliptic} and the sixteen curves in \cite[Table~3]{xu2026torsiongrowthrationalelliptic} are precisely the curves appearing here. Building on their work, we prove in Theorems~\ref{theorem:CMclass,p=3} and~\ref{theorem:CMclass,p=2} that, up to \(\mathbb{Q}\)-isomorphism, these lists exhaust all elliptic curves over \(\mathbb{Q}\) satisfying the corresponding CM and good-reduction conditions.

Combine these results, we classify all triples \((E,K,p)\) for which \(E(K_\infty)_{\mathrm{tors}}\) is infinite.
\end{remark}

\begin{remark}
    The methods used in this article rely on elementary algebraic number theory.
    A key technique is the combination of the Weil pairing (Section~\ref{subsubsection:Weil pairing}) with the growth pattern of torsion subgroups (Section~\ref{subsubsection:growthpattern}).
    For an intermediate field \(F\) in \(K_\infty/K\) with \(F/\mathbb{Q}\) Galois, this combination forces \(E(F)[q^\infty]\) to be a cyclic group in most cases. In such a situation, our method is fully effective.

    When \(K=\mathbb{Q}(\sqrt{-3})\) with \(p=3\), or  \(p=2\), more roots of unity are contained in a \(\mathbb{Z}_p\)-extension and we study the Galois representations associated with \(E\). In these cases, we also use CM theory to study when \(E(K_\infty)_{\mathrm{tors}}\) is infinite. We also try to give the growth pattern in these cases. However, for \(p=q=2\), the growth pattern is less explicit than \(p=3\). 
\end{remark}

\subsection{Notations}

Throughout the paper we use the following conventions.

\begin{itemize}
    \item \(E\) denotes an elliptic curve defined over \(\mathbb{Q}\).
    \item \(K\) is a quadratic number field. When \(p\) is fixed, \(K_\infty\) denotes the compositum of all \(\mathbb{Z}_p\)-extensions of \(K\).
    \item For \(n\ge0\), \(K_\infty^{(n)}\) denotes the compositum of the \(n\)-th layers of all \(\mathbb{Z}_p\)-extensions of \(K\), with \(K_\infty^{(0)}=K\).
    \item For an algebraic extension \(L/\mathbb{Q}\), \(E(L)_{\mathrm{tors}}\) is the torsion subgroup of the Mordell-Weil group \(E(L)\).  
          For a prime \(q\), \(E(L)_{\mathrm{tors}}[q^\infty]\) denotes its \(q\)-primary part, and \(E(L)_{\mathrm{tors}}[n]\) the group of \(n\)-torsion points.
    \item \(\mu_n\) is the group of \(n\)-th roots of unity. \(\mu(L)\) denotes the set of all roots of unity contained in a field \(L\).
    \item \(\mathbb{Q}_{\infty,p}\) stands for the cyclotomic \(\mathbb{Z}_p\)-extension of \(\mathbb{Q}\).
    \item When \(K\) is an imaginary quadratic field, \(K_{cyc}\) and \(K_{anti}\) denote its cyclotomic and anticyclotomic \(\mathbb{Z}_p\)-extensions, respectively.
\end{itemize}

\subsection*{Acknowledgment}
The author thanks Dr. Ruichen Xu for helpful conversations. The author thanks the anonymous referees for pointing out an error in the earlier draft. The author thanks Shisong Xu and Haoyang Yuan for their work on this topic. This work was carried out almost simultaneously with theirs and overlaps in part. The author also acknowledges the Morningside Center of Mathematics, Chinese Academy of Sciences, where part of this work was carried out.

\section{Preliminaries}
We begin by recalling the basic notions that will be used throughout this article. In particular, we collect several lemmas and propositions that are needed in the proof of the main theorem.

\subsection{\(\mathbb{Z}_p\)-extensions of number fields}\label{subsection:zpextension}

Let \(F\) be a number field. A \(\mathbb{Z}_p\)-extension \(L/F\) is a Galois extension such that
\[
\Gal{L}{F} \cong \mathbb{Z}_p.
\]
Equivalently, there exists a tower of field extensions
\[
F=F_0 \subset F_1 \subset F_2 \subset \cdots \subset F_n \subset \bigcup_{n\ge 1} F_n = L
\]
such that \(\Gal{F_n}{F} \cong \mathbb{Z}/p^n \mathbb{Z}\) for each \(n\ge 1\).

Every number field \(F\) admits at least one \(\mathbb{Z}_p\)-extension, namely the cyclotomic \(\mathbb{Z}_p\)-extension \(F_{cyc}\). It can be constructed as follows. Let \(F(\mu_{p^\infty})\) denote the field obtained by adjoining all \(p\)-power roots of unity to \(F\). The Galois group \(\Gal{F(\mu_{p^\infty})}{F}\) is a subgroup of 
\[
\Gal{\mathbb{Q}(\mu_{p^\infty})}{\mathbb{Q}} \cong \mathbb{Z}_p^\times \cong \Delta \times \mathbb{Z}_p,
\]
where \(\Delta\) is a finite group. Hence \(\Gal{F(\mu_{p^\infty})}{F} \cong \Delta' \times p^k \mathbb{Z}_p\), and the cyclotomic \(\mathbb{Z}_p\)-extension \(F_{cyc}\) is the subfield fixed by \(\Delta'\) via Galois theory.

For a general number field \(F\), there may exist other \(\mathbb{Z}_p\)-extensions. In fact, we have the following standard result, see \cite[Theorem 13.4, Corollary 5.32]{Washington1997}.

\begin{proposition}\label{prop:allZpext}
    Let \(r_2\) denote the number of pairs of complex embeddings of \(F\), and let \(F_\infty\) be the compositum of all \(\mathbb{Z}_p\)-extensions of \(F\). Then
    \[
    \Gal{F_\infty}{F} \cong \mathbb{Z}_p^{\, r_2 + 1 + \delta},
    \]
    where \(\delta\) is a constant conjectured to be \(0\) (the Leopoldt conjecture). Moreover, if \(F/\mathbb{Q}\) is abelian, then \(\delta = 0\).
\end{proposition}

Now let \(K\) be a quadratic field. Since \(K/\mathbb{Q}\) is abelian, Proposition \ref{prop:allZpext} implies that the Galois group of the compositum \(K_\infty\) of all \(\mathbb{Z}_p\)-extensions of \(K\) over \(K\) is isomorphic to \(\mathbb{Z}_p\) if \(K\) is real quadratic, and to \(\mathbb{Z}_p^2\) if \(K\) is imaginary quadratic. In both cases, \(K_\infty/\mathbb{Q}\) is Galois. Indeed, for any \(\sigma \in \Gal{\overline{\mathbb{Q}}}{\mathbb{Q}}\) and any \(\mathbb{Z}_p\)-extension \(L/K\), the conjugate \(\sigma(L)/\sigma(K)\) is also a \(\mathbb{Z}_p\)-extension. Since \(K/\mathbb{Q}\) is quadratic, we have \(\sigma(K) = K\), so \(\sigma(L)\) is again a \(\mathbb{Z}_p\)-extension of \(K\), hence \(\sigma(L) \subset K_\infty\). As \(L\) ranges over all \(\mathbb{Z}_p\)-extensions of \(K\), it follows that \(\sigma(K_\infty) \subset K_\infty\). Since \(\sigma\) was arbitrary, applying the same argument to \(\sigma^{-1}\) yields \(\sigma^{-1}(K_\infty) \subseteq K_\infty\), and therefore \(\sigma(K_\infty) = K_\infty\). Thus \(K_\infty/\mathbb{Q}\) is Galois.

Since \(K_\infty/\mathbb{Q}\) is Galois, the group \(\Gal{K}{\mathbb{Q}}\) acts on \(\Gal{K_\infty}{K}\) by conjugation. Hence \(\Gal{K_\infty}{K}\) is a \(\mathbb{Z}_p[\Gal{K}{\mathbb{Q}}]\)-module, and therefore decomposes into eigenspaces for \(p>2\). If \(p>2\), then \(2\) is invertible in \(\mathbb{Z}_p\), so \(\Gal{K_\infty}{K}\) splits into its \(+1\)- and \(-1\)-parts.

If \(K\) is real, then by Proposition \ref{prop:allZpext} it has a unique \(\mathbb{Z}_p\)-extension, so \(K_\infty = K_{cyc} = \mathbb{Q}_{cyc}K\). Hence \(K_\infty/\mathbb{Q}\) is abelian.

If \(K\) is imaginary, then \(K_{cyc}/\mathbb{Q}\) is also abelian, so the \(+1\)-part is nontrivial. Moreover, by the Kronecker--Weber theorem, the maximal pro-\(p\) abelian extension of \(\mathbb{Q}\) has \(\mathbb{Z}_p\)-rank \(1\) unramified outside \(p\). Hence \(K_\infty/\mathbb{Q}\) is not abelian, so the \(-1\)-part is also nontrivial. Since \(\Gal{K_\infty}{K}\) has \(\mathbb{Z}_p\)-rank \(2\), it decomposes as
\[
\Gal{K_\infty}{K} = \Gal{K_\infty}{K}^+ \oplus \Gal{K_\infty}{K}^-,
\]
with each factor of \(\mathbb{Z}_p\)-rank \(1\). The subfield fixed by the \(\Gal{K_\infty}{K}^-\) is the cyclotomic \(\mathbb{Z}_p\)-extension \(K_{cyc}\) discussed above. On the other hand, the subfield fixed by the \(\Gal{K_\infty}{K}^+\) is called the anticyclotomic \(\mathbb{Z}_p\)-extension of \(K\), denoted by \(K_{anti}\).

We note that for \(p>2\), although an imaginary quadratic field admits infinitely many \(\mathbb{Z}_p\)-extensions, only two of them are Galois over \(\mathbb{Q}\), namely the cyclotomic and the anticyclotomic \(\mathbb{Z}_p\)-extensions. The only \(\mathbb{Z}_p\)-extension which is abelian over \(\mathbb{Q}\) is the cyclotomic one. Moreover, if we consider the subfields in \(K_\infty\), the only subfields that are abelian over \(\mathbb{Q}\) are the subfields of the cyclotomic \(\mathbb{Z}_p\)-extension.

We remark that for \(p=2\), there is no anticyclotomic extension in the sense of the above discussion since \(2\) is not invertible in \(\mathbb{Z}_2\). Nevertheless, we still have the notions of cyclotomic and anticyclotomic extensions, which will be discussed in Section~\ref{section:p=2}.

The following fact will be used repeatedly in this article, see \cite[Proposition 13.2]{Washington1997}.

\begin{proposition}\label{prop:unramifiedoutsidep}
    Let \(L/F\) be a \(\mathbb{Z}_p\)-extension, and let \(\tilde{l}\) be a prime (archimedean or non-archimedean) of \(F\) which does not lie above \(p\). Then \(L/F\) is unramified at \(\tilde{l}\).
\end{proposition}

\subsection{Restrictions on torsion subgroups of elliptic curves}
We collect in this subsection some elementary but essential facts that will be used to restrict the torsion subgroups of elliptic curves.

\subsubsection{Weil pairing}\label{subsubsection:Weil pairing}

Let \(E\) be an elliptic curve defined over a number field \(F\). For \(n \ge 1\), let \(E[n]\) denote the subgroup of \(n\)-torsion points of \(E(\overline{F})\). It is a standard fact that \(E[n] \cong (\mathbb{Z}/n\mathbb{Z})^2\), see \cite[III.6.1]{Silverman2009}. Associated to this is the Weil pairing
\[
e_n: E[n] \times E[n] \longrightarrow \mu_n,
\]
which is bilinear, non-degenerate, and compatible with the Galois action. For any algebraic extension \(L/F\), restricting to \(L\)-rational points yields
\[
e_n(L): E(L)[n] \times E(L)[n] \longrightarrow \mu_n(L).
\]
A key consequence of the Galois equivariance of \(e_n\) is the following, see \cite[Corollary 8.1.1]{Silverman2009}.

\begin{lemma}\label{lemma:rootsofunity}
    Let \(E/F\) be an elliptic curve over a number field \(F\), and let \(L/F\) be an algebraic extension. If \(E[n] \subseteq E(L)\), then \(\mu_n \subset L\).
\end{lemma}

\subsubsection{Isogenies}

Let \(E\) be an elliptic curve defined over a number field \(F\), and let \(L/F\) be an algebraic extension. An \(n\)-isogeny \(\phi\) of \(E\) defined over \(L\) is an \(L\)-morphism
\[
\phi: E \longrightarrow E'
\]
to another elliptic curve \(E'/L\) such that \(\deg(\phi)=n\). We say that \(E\) admits an \(n\)-isogeny over \(L\) if such an isogeny exists.

If \(E[n]\) contains a \(\Gal{\overline{F}}{L}\)-stable cyclic subgroup \(C\) of order \(n\), then the quotient map \(\phi_C: E \rightarrow E/C\) is an \(n\)-isogeny over \(L\). In particular, if \(P \in E(L)\) has order \(n\), then \(E\) admits an \(n\)-isogeny over \(L\).

A natural question is whether the existence of rational points over \(L\) forces the existence of an isogeny over the base field \(F\). The following lemma provides an answer.

\begin{lemma}\label{lemma:admits-n-isogeny}
    Let \(E\) be an elliptic curve over a field \(F\) of characteristic \(0\), and let \(L/F\) be a Galois extension. If
    \[
    E(L)_{\mathrm{tors}} \cong \mathbb{Z}/m\mathbb{Z} \times \mathbb{Z}/mn\mathbb{Z},
    \]
    then \(E\) admits an \(n\)-isogeny over \(F\).
\end{lemma}

\begin{proof}
    The proof is essentially contained in the proof of \cite[Lemma 2.7]{chou2019}, which treats the case \(F=\mathbb{Q}\). The argument extends verbatim to any field of characteristic \(0\).
\end{proof}

Moreover, the following results concerning isogenies can be related to torsion subgroups by Lemma~\ref{lemma:admits-n-isogeny}.

\begin{theorem}[\cite{Mazur1978rational} Theorem 1, \cite{kenku1982onthe} Theorem 1]\label{theorem:isogeniesoverQ}
    If \( E/\mathbb{Q} \) has an \( n \)-isogeny, then \( n \leq 19 \) or \( n \in \{21, 25, 27, 37, 43, 67, 163\} \).
\end{theorem}

\subsubsection{Growth pattern of torsion subgroups in \(\mathbb{Z}_p\)-extensions}\label{subsubsection:growthpattern}

Let \(E\) be an elliptic curve defined over a number field \(F\), and let \(L/F\) be a \(\mathbb{Z}_p\)-extension with intermediate subfields \(F_n\) for \(n\ge 1\). For a prime \(q\), denote by
\[
M(n,q) := E(F_n)_{\mathrm{tors}}[q^\infty]
\]
the \(q^\infty\)-part of the torsion subgroup of \(E(F_n)\). The abelian group \(M(n,q)\) carries a natural structure as a \(\mathbb{Z}_q[\Gal{F_n}{F}]\)-module.

In the case where \(q \neq p\), we have the following algebraic fact concerning the structure of the group ring \(\mathbb{Z}_q[\Gal{F_n}{F}]\).

\begin{lemma}\label{lemma:structure}
    Suppose \(q \neq p\). Then
    \[
    \mathbb{Z}_q[\Gal{F_n}{F}] \cong \prod_{j=0}^n \prod_{i=1}^{m_j} \mathcal{O}_{j,i},
    \]
    where each \(\mathcal{O}_{j,i}\) is an unramified extension of \(\mathbb{Z}_q\), with residue field of size \(q^{d_j}\). Here,
    \[
    d_0 = 1,\quad m_0 = 1,
    \]
    and for \(j \ge 1\),
    \[
    d_j := \ord_{p^j}(q) = \min\{ d \ge 1 : q^d \equiv 1 \pmod{p^j} \}, \qquad
    m_j := \frac{\varphi(p^j)}{d_j}.
    \]
\end{lemma}

\begin{proof}
    Consider the finite field \(\mathbb{F}_q\) with \(q\) elements. Since \(\Gal{F_n}{F}\) is cyclic of order \(p^n\), we have
    \[
    \mathbb{F}_q[\Gal{F_n}{F}] \cong \mathbb{F}_q[X]/(X^{p^n}-1).
    \]
    Because \((q,p)=1\), the polynomial \(X^{p^n}-1\) has no repeated roots over \(\mathbb{F}_q\) and factors as
    \[
    X^{p^n}-1 = (X-1)\prod_{j=1}^{n} \Phi_{p^j}(X),
    \]
    where \(\Phi_{p^j}(X)\) denotes the \(p^j\)-th cyclotomic polynomial. Setting \(\Phi_1(X)=X-1\) and \(p^0=1\), we may write uniformly
    \[
    X^{p^n}-1 = \prod_{j=0}^{n} \Phi_{p^j}(X).
    \]
    Each \(\Phi_{p^j}(X)\) decomposes into \(m_j\) distinct irreducible factors of degree \(d_j\), with \(d_0=1,\ m_0=1\) and for \(j\ge 1\), \(d_j\) and \(m_j\) as defined in the lemma. Consequently,
    \[
    \mathbb{F}_q[\Gal{F_n}{F}] \cong \prod_{j=0}^{n} \prod_{i=1}^{m_j} \mathbb{F}_{q^{d_j}}.
    \]
    The desired decomposition over \(\mathbb{Z}_q\) is then obtained by applying Hensel's lemma to lift the idempotents from the residue field. This completes the proof.
\end{proof}

Since \(M(n,q)\) is a \(\mathbb{Z}_q[\Gal{F_n}{F}]\)-module, Lemma~\ref{lemma:structure} shows that \(M(n,q)\) splits into summands, each of which is an \(\mathcal{O}_{j,i}\)-module. As a subgroup of \(E[q^\infty] \cong (\mathbb{Q}_q/\mathbb{Z}_q)^2\), the \(\mathbb{Z}_q\)-module \(M(n,q)\) can be generated by at most two elements. With these tools we now examine how the \(q\)-primary torsion can grow when passing from \(F_n\) to \(F_{n+1}\).

\begin{lemma}\label{lemma:growthpattern=ellnotp}
    Let \(E\) be an elliptic curve defined over a number field \(F\), and let \(L/F\) be a \(\mathbb{Z}_p\)-extension with intermediate subfields \(F_n\) for \(n\ge 1\), and set \(F_0=F\). Let \(q\) be a prime such that \(q \neq p\). Then for each \(n\ge 0\), exactly one of the following holds:
    \begin{enumerate}
        \item \(E(F_{n+1})[q^\infty] = E(F_n)[q^\infty]\);
        \item \(E(F_{n+1})[q^\infty] \cong E(F_n)[q^\infty] \times \mathbb{Z}/q^a\mathbb{Z}\) for some integer \(a \ge 1\);
        \item \(E(F_{n+1})[q^\infty] \cong E(F_n)[q^\infty] \times \mathbb{Z}/q^a\mathbb{Z} \times \mathbb{Z}/q^b\mathbb{Z}\) for some integers \(a,b \ge 1\). In this case, necessarily \(E(F_n)[q^\infty] = 0\).
    \end{enumerate}
\end{lemma}

\begin{proof}
Let \(G_n=\Gal{F_n}{F}\) and \(H_n:=\Gal{F_{n+1}}{F_n}\). Then we have \(
G_n\cong \mathbb{Z}/p^n\mathbb{Z}, H_n\cong \mathbb{Z}/p\mathbb{Z},
\)
and the short exact sequence
\[
1\longrightarrow H_n\longrightarrow G_{n+1}\longrightarrow G_n\longrightarrow 1.
\]
Let \(M_n:=E(F_n)[q^\infty]\). Since \(q\neq p\), the element \(p\) is a unit in \(\mathbb{Z}_q\). Hence the element
\[
e_n:=\frac{1}{p}\sum_{h\in H_n}h
\in \mathbb{Z}_q[G_{n+1}]
\]
is well-defined and satisfies \(e_n^2=e_n\). Moreover, for any
\(\mathbb{Z}_q[G_{n+1}]\)-module \(M\), we have 
\[
e_nM=M^{H_n}.
\]
Therefore we have 
\[M_{n+1} = e_nM_{n+1}\oplus (1-e_n)M_{n+1}.\]
Since \(M_{n+1}^{H_n}=E(F_{n+1})[q^\infty]^{H_n}=E(F_n)[q^\infty]=M_n,\)
we obtain
\[
M_{n+1}=M_n\oplus M_{\mathrm{new}},
\qquad
M_{\mathrm{new}}:=(1-e_n)M_{n+1}.
\]

By Lemma~\ref{lemma:structure},
\[
\mathbb{Z}_q[G_{n+1}]
\cong
\mathbb{Z}_q[G_n]
\times
\prod_{i=1}^{m_{n+1}}\mathcal{O}_{n+1,i},
\]
where each \(\mathcal{O}_{n+1,i}\) is an unramified extension of
\(\mathbb{Z}_q\) of degree \(d_{n+1}=\ord_{p^{n+1}}(q).\)
Thus
\[
M_{\mathrm{new}}
=
\bigoplus_{i=1}^{m_{n+1}} M_i,
\]
where each \(M_i\) is a finite torsion
\(\mathcal{O}_{n+1,i}\)-module.

By the structure theorem for finite torsion modules over a discrete
valuation ring,
\[
M_i
\cong
\bigoplus_{r=1}^{t_i}
\mathcal{O}_{n+1,i}/q^{a_{i,r}}\mathcal{O}_{n+1,i}.
\]
Since \(\mathcal{O}_{n+1,i}\) is unramified of degree \(d_{n+1}\) over
\(\mathbb{Z}_q\), we have
\[
\mathcal{O}_{n+1,i}/q^a\mathcal{O}_{n+1,i}
\cong
(\mathbb{Z}/q^a\mathbb{Z})^{d_{n+1}}
\]
as \(\mathbb{Z}_q\)-modules. Hence every nonzero cyclic
\(\mathcal{O}_{n+1,i}\)-summand requires at least \(d_{n+1}\) generators
as a \(\mathbb{Z}_q\)-module.

On the other hand, \( M_{n+1}\subset E[q^\infty] \cong (\mathbb{Q}_q/\mathbb{Z}_q)^2, \) so \(M_{n+1}\) can be generated by at most two elements as a
\(\mathbb{Z}_q\)-module.

We now consider the following three cases.

If \(d_{n+1}\ge 3\), then any nonzero \(M_{\mathrm{new}}\) would require
at least three generators. Hence \( M_{\mathrm{new}}=0. \) It follows that \(M_{n+1}=M_n,\) which is case~1.

Suppose that \(d_{n+1}=2\). Then every nonzero cyclic
\(\mathcal{O}_{n+1,i}\)-summand requires two generators as a
\(\mathbb{Z}_q\)-module. Since \(M_{n+1}\) can be generated by at most
two elements, \(M_{\mathrm{new}}\) can contain at most one such cyclic
summand. Thus, if \(M_{\mathrm{new}}\neq0\),
\[
M_{\mathrm{new}}
\cong
\mathcal{O}_{n+1,i}/q^a\mathcal{O}_{n+1,i}
\cong
(\mathbb{Z}/q^a\mathbb{Z})^2
\]
for some \(i\) and some \(a\ge1\). In this case \(M_n=0\), since otherwise  \(M_{n+1}\) requires at least three generators. Therefore
\[
M_{n+1}
\cong
(\mathbb{Z}/q^a\mathbb{Z})^2,
\]
which is case~3 with \(a=b\). If \(M_{\mathrm{new}}=0\), we are in
case~1.

Finally, suppose that \(d_{n+1}=1\). Then each
\(\mathcal{O}_{n+1,i}\) is equal to \(\mathbb{Z}_q\), and therefore
\[
M_{\mathrm{new}}
\cong
\bigoplus_r \mathbb{Z}/q^{a_r}\mathbb{Z}.
\]
Since \(M_{n+1}\) requires at most two generators, the new part has at
most two cyclic summands. Thus the possibilities are
\[
M_{\mathrm{new}}=0,
\]
\[
M_{\mathrm{new}}
\cong
\mathbb{Z}/q^a\mathbb{Z},
\]
or
\[
M_{\mathrm{new}}
\cong
\mathbb{Z}/q^a\mathbb{Z}
\times
\mathbb{Z}/q^b\mathbb{Z}.
\]
The first possibility gives case~1, while the second gives case~2. In
the third possibility, \(M_n\) must be trivial, since otherwise
\(M_{n+1}=M_n\oplus M_{\mathrm{new}}\) would require at least three
generators. Hence we obtain case~3.

These possibilities exhaust all cases, and the proof is complete.
\end {proof}

Now we turn to the case where \(q = p\). In this situation, the group algebra \(\mathbb{F}_p[\Gal{F_n}{F}]\) is no longer semisimple, so the structure theorem used in the previous lemma does not apply. Nevertheless, we have the following result (see \cite[Exercises 1.13]{greenberg2001}).

\begin{lemma}\label{lemma:grothpatternp=ell}
    Let \(E\) be an elliptic curve defined over a number field \(F\), and let \(L/F\) be a \(\mathbb{Z}_p\)-extension with intermediate subfields \(F_n\) for \(n\ge 1\). If \(E(F)_{\mathrm{tors}}[p^\infty] = 0\), then \(E(L)_{\mathrm{tors}}[p^\infty] = 0\).
\end{lemma}

\begin{proof}
    We show that \(E(F_n)[p] = 0\) for every \(n \ge 0\). Suppose, for contradiction, that there exists some \(n \ge 1\) with \(E(F_n)[p] \neq 0\), and choose the minimal such \(n\). The Galois group \(G := \Gal{F_n}{F}\) is a finite \(p\)-group acting on the non-trivial finite \(p\)-group \(V := E(F_n)[p]\).

    By the orbit-stabilizer theorem, the number of fixed points satisfies \(|V^G| \equiv |V| \pmod p\). Since \(0 \in V^G\) and \(|V|\) is a positive power of \(p\), we must have \(V^G \neq \{0\}\). But
    \[
    V^G = E(F_n)[p]^{\Gal{F_n}{F}} = E(F)[p] = 0,
    \]
    a contradiction. Hence \(E(F_n)[p] = 0\) for every \(n \ge 0\), and consequently \(E(L)_{\mathrm{tors}}[p^\infty] = 0\).
\end{proof}

The following lemma also deals with the case where \(E(F)_{\mathrm{tors}}[p]\) is non-trivial.

\begin{lemma}\label{lemma:grothpatternp=ell,depth}
    Let \(E\) be an elliptic curve defined over a number field \(F\), and let \(L/F\) be a \(\mathbb{Z}_p\)-extension with intermediate subfields \(F_n\) for \(n\ge 1\), and set \(F_0=F\). Suppose that \(E(F)_{\mathrm{tors}}[p^\infty]=\mathbb{Z}/p^k\mathbb{Z}\) with \(k \ge 1\). Also suppose that
    \(E(L)_{\mathrm{tors}}[p^\infty]\cong \mathbb{Z}/p^s\mathbb{Z}\) with \(s\ge k \) or \(\mathbb{Q}_p/\mathbb{Z}_p\). In the latter case, we set \(s=\infty\).

    \begin{itemize}
        \item Suppose that \(p\) is odd. Then  \(E(F_n)_{\mathrm{tors}}[p^\infty]\cong \mathbb{Z}/p^{\min \{k+n,s\} } \mathbb{Z}\) for any \(n\ge 0\).

        \item Suppose that \(p=2\). 
            \begin{itemize}
                \item If \(k\ge 2\), then \(E(F_n)_{\mathrm{tors}}[p^\infty]\cong \mathbb{Z}/2^{\min \{k+n,s\} } \mathbb{Z}\) for any \(n\ge 0\).
                \item If \(k=1\) and \(E(F_1)_{\mathrm{tors}}[p^\infty]\cong \mathbb{Z}/2^{r_1} \mathbb{Z}\), then \(E(F_n)_{\mathrm{tors}}[p^\infty]\cong \mathbb{Z}/2^{\min \{r_1+n-1,s\} } \mathbb{Z}\) for any \(n\ge 1\).
            \end{itemize}
    \end{itemize}
    
\end{lemma}

\begin{proof}
    Set \(M=E(L)_{\mathrm{tors}}[p^\infty]\) and \(\Gamma=\Gal{L}{F}\). Let \(\gamma\) be a topological generator of \(\Gamma\).

    \textbf{Case 1: odd \(p\).} We first prove the case where \(p\) is odd and \(s\) is finite. Since \(\Gamma\) acts on \(M\), there is a map
    \[
    \Gamma \rightarrow \Aut(M)\cong (\mathbb{Z}/p^s \mathbb{Z})^\times.
    \]
    Let \(Q\) be a generator of \(M\), and write \(\gamma(Q)=aQ\) for \(a\in (\mathbb{Z}/p^s\mathbb{Z})^\times\). If \(s=k\), then \(M=M^\Gamma\), so the action is trivial and the conclusion is immediate. Suppose that \(s>k\). Since \(M^\Gamma=E(F)_{\mathrm{tors}}[p^\infty] \cong \mathbb{Z}/p^k \mathbb{Z}\), the point \(p^{s-k}Q\) is fixed by \(\Gamma\). Hence \(p^{s-k}Q=\gamma(p^{s-k}Q)=ap^{s-k}Q\), so \(p^s\mid(a-1)p^{s-k}\). Moreover, the fixed subgroup has order exactly \(p^k\), and therefore \(\ord_p(a-1)=k\).

    For \(\Gamma_n=\Gamma^{p^n}\), we have
    \[
    M^{\Gamma_n}=\ker(a^{p^n}-1 :M \rightarrow M).
    \]
    By the lifting-the-exponent lemma, \(\ord_p(a^{p^n}-1)=k+n\). Hence, \(M^{\Gamma_n}\) is generated by \(p^{s-k-n}Q\) when
    \(k+n<s\) and \(Q\) when \(k+n \ge s\). This shows the case.

    We now prove the case \(s=\infty\). Suppose that \(M\cong \mathbb Q_p/\mathbb Z_p\). Choose a compatible system of generators \(\{Q_m\}_{m\ge1}\), where \(Q_m\in M[p^m]\) and \(pQ_{m+1}=Q_m\). Then the action of \(\gamma\) is given by \(\gamma(Q_m)=aQ_m\) for some \(a\in\mathbb Z_p^\times\), independent of \(m\). As in the previous case, \(M^\Gamma=E(F)_{\mathrm{tors}}[p^\infty]\cong\mathbb{Z}/p^k\mathbb{Z}\), so \(\ord_p(a-1)=k\) and \(\ord_p(a^{p^n}-1)=k+n\). Therefore \(M^{\Gamma_n}\cong\mathbb{Z}/p^{k+n}\mathbb{Z}\) for \(n\ge0\).

    \textbf{Case 2: \(p=2\).} For \(p=2\), we similarly consider the action of \(\Gamma\) on \(M\). If \(s<\infty\), choose a generator \(Q\) of \(M\) and write \(\gamma(Q)=aQ\) for some \(a\in\mathbb {Z}_2^\times\). The case \(s=\infty\) is analogous. If \(s=k\), then \(M=M^\Gamma\) and there is nothing to prove. If \(s>k\), the equality \(M^\Gamma=E(F)_{\mathrm{tors}}[2^\infty]\cong\mathbb{Z}/2^k\mathbb{Z}\) gives \(\ord_2(a-1)=k.\)

    Suppose first that \(k\ge2\). Then \(a\equiv1\pmod4\), so \(\ord_2(a+1)=1\). Hence, by the \(2\)-adic version of the Lift The Exponent lemma,
    \(
    \ord_2(a^{2^n}-1)=\ord_2(a-1)+\ord_2(a+1)+n-1=k+n
    \)
    for \(n\ge1\). Thus
    \[
    M^{\Gamma_n}=\ker(a^{2^n}-1:M\rightarrow M)\cong\mathbb{Z}/2^{\min\{k+n,s\}}\mathbb{Z}.
    \]
    
    Now suppose that \(k=1\). Then \(a\equiv-1\pmod4\). Let \(b=\ord_2(a+1)\). If \(a=-1\), then \(\gamma^2\) acts trivially on \(M\), so \(M^{\Gamma_n}=M\) for every \(n\ge1\). Otherwise \(b\ge2\), and for \(n\ge1\), \(\ord_2(a^{2^n}-1)=\ord_2(a^2-1)+n-1=\ord_2(a-1)+\ord_2(a+1)+n-1=b+n\). Therefore \(M^{\Gamma_n}\cong\mathbb Z/2^{\min\{b+n,s\}}\mathbb Z\) for \(n\ge 1\).

    In particular, if \(E(F_1)_{\mathrm{tors}}[2^\infty]\cong\mathbb Z/2^{r_1}\mathbb Z,\) then \(r_1=\min\{b+1,s\}\), and hence \( E(F_n)_{\mathrm{tors}}[2^\infty] \cong \mathbb Z/2^{\min\{r_1+n-1,s\}}\mathbb Z\), for \(n \ge 1\).

    The case \(s=\infty\) is obtained by omitting the truncation by \(s\).
\end{proof}

\begin{lemma}\label{lemma:fixed-point-growth}
Let \(M\) be a finite abelian \(p\)-group and let \(g\in\Aut(M)\) be a non-trivial element of \(p\)-power order. Then
\[
M^g\subsetneq M^{g^p}.
\]
\end{lemma}

\begin{proof}
Let \(A=g-1\) and \(S=1+g+\cdots+g^{p-1}\), so that \(g^p-1=AS=SA\). Suppose that \(M^g=M^{g^p}\). Then \(\ker A=\ker(AS)\), and the finiteness of \(M\) gives
\[
|\im A|=|\im(AS)|.
\]
Let \(N=\im A\). Since \(\im(AS)=S(N)\subseteq N\), the endomorphism \(S\) is bijective on \(N\).

If \(N\neq 0\), then \(N/pN \neq 0\). The action of \(g-1\) on this \(\mathbb{F}_p\)-vector space is nilpotent because \(g\) has \(p\)-power order. Moreover,
\[
1+X+\cdots+X^{p-1}=(X-1)^{p-1}
\]
over \(\mathbb{F}_p\). Thus \(S\) acts nilpotently on \(N/pN\), contradicting the fact that the bijection \(S:N\to N\) induces a bijection on \(N/pN\). Hence \(N=0\), so \(g=1\), a contradiction.
\end{proof}

\subsubsection{The field \(\mathbb{Q}(P)\).} Now suppose that \(P\) is a torsion point of \(E(\overline{\mathbb{Q}})\) of order \(N\), and let \(\mathbb{Q}(P)\) be the field obtained by adjoining the coordinates of \(P\) to \(\mathbb{Q}\). In general, \(\mathbb{Q}(P)\) is not Galois over \(\mathbb{Q}\). If we suppose that the group \(\langle P \rangle\) is \(G_\mathbb{Q}\)-stable in \(E[N]\), then we have a character associated with the action of \(G_{\mathbb{Q}}\) on \(\langle P\rangle\) 
\[
\rho_P: G_{\mathbb{Q}}\longrightarrow \Aut(\langle P\rangle) \cong (\mathbb{Z}/N \mathbb{Z})^\times,
\]
whose fixed field is \(\mathbb{Q}(P)\). Hence \(\mathbb{Q}(P)\) is Galois over \(\mathbb{Q}\). Moreover, since \(\Gal{\mathbb{Q}(P)}{\mathbb{Q}}\) is a subgroup of \((\mathbb{Z}/N \mathbb{Z})^\times\), \(\mathbb{Q}(P)\) is an abelian extension.

Moreover, if \(K\) is another abelian extension of \(\mathbb{Q}\) (a quadratic field, for example), then \(K(P)=K \mathbb{Q}(P)\) is also abelian over \(\mathbb{Q}\). Hence we have the following lemma.

\begin{lemma}\label{lemma:KQPabelian}
    Let \(E\) be an elliptic curve over \(\mathbb{Q}\), let \(K\) be a quadratic field, and let \(P\) be a torsion point such that \(\langle P\rangle\) is \(G_\mathbb{Q}\)-stable. Then \(K\mathbb{Q}(P)/\mathbb{Q}\) is abelian.
\end{lemma}

\subsection{Quadratic twists of rational elliptic curves}

Let \(E/\mathbb{Q}\) be an elliptic curve and let \(K/\mathbb{Q}\) be a quadratic field.
Denote by \(\chi_K : G_{\mathbb{Q}} \to \{\pm 1\}\) the quadratic character associated with \(K=\mathbb{Q}(\sqrt{d})\) (where \(d\neq 0,1\) is a square-free integer), so that \(\ker \chi_K = G_K\).

The \emph{quadratic twist of \(E\) by \(K\)} (or by \(\chi_K\)) is the elliptic curve \(E^{(K)}/\mathbb{Q}\) given by
\[
E^{(K)} : d y^2 = x^3 + A x + B,
\qquad\text{where}\quad E : y^2 = x^3 + A x + B.
\]

\begin{proposition}\label{prop:twist-basic}
\leavevmode
\begin{enumerate}
    \item \(E^{(K)}\) becomes isomorphic to \(E\) over \(K\) via
    \[
    \phi: E \times_{\mathbb{Q}} K \longrightarrow E^{(K)} \times_{\mathbb{Q}} K,
    \qquad (x,y) \longmapsto (x, y/\sqrt{d}),
    \]
    which is an isomorphism defined over \(K\). In particular, for an algebraic extension \(L/K\), \(E(L)\cong E^{(K)}(L)\).
    \item For every \(n \ge 1\), the mod \(n\) Galois representations satisfy
    \[
    \rho_{E^{(K)}, n} \cong \chi_K \otimes \rho_{E,n},
    \]
    where \(\rho_{E,n} : G_{\mathbb{Q}} \to \GL(E[n])\).
\end{enumerate}
\end{proposition}

Let \(P\in E(\overline{\mathbb{Q}})\) be a point of order \(q^n\) (for a prime \(q\) and \(n\ge 1\)) such that the cyclic subgroup \(\langle P\rangle\) is \(G_{\mathbb{Q}}\)-stable. Recall that the action of \(G_{\mathbb{Q}}\) on \(\langle P\rangle\) yields a character
\[
\rho_P: G_{\mathbb{Q}}\longrightarrow \Aut(\langle P\rangle) \cong (\mathbb{Z}/q^n \mathbb{Z})^\times,
\]
whose fixed field is \(\mathbb{Q}(P)\). Consequently,
\[
[\mathbb{Q}(P):\mathbb{Q}] \mid \varphi(q^n)=(q-1)q^{n-1}.
\]

The following two lemmas, treating the cases \(q=p\) and \(q\neq p\) respectively, will be used repeatedly.

\begin{lemma}\label{lemma:p-twist}
    Let \(p\) be an odd prime, \(E/\mathbb{Q}\) an elliptic curve, and \(K\) a quadratic field. 
    Let \(K_{\infty}/K\) be the compositum of all \(\mathbb{Z}_p\)-extensions of \(K\), and let \(L/K\) be an intermediate finite field with \(L/\mathbb{Q}\) Galois. 
    Suppose that \(P\in E(L)\) is a torsion point of order \(p^n\) (\(n\ge 1\)) with \(P\notin E(K)\), and that \(\langle P\rangle\) is \(G_{\mathbb{Q}}\)-stable.
    Then one of the following holds:
    \begin{enumerate}
        \item If \(K\not\subset \mathbb{Q}(P)\), then \(\mathbb{Q}(P)\subseteq \mathbb{Q}_{\infty,p}\), the cyclotomic \(\mathbb{Z}_p\)-extension of \(\mathbb{Q}\). In particular, \(P\in E(\mathbb{Q}_{\infty,p})\).
        \item If \(K\subset \mathbb{Q}(P)\), then \(\mathbb{Q}(P)\) contains a unique subfield \(F_p\) which is a \(p\)-power degree extension of \(\mathbb{Q}\) unramified outside \(p\), and \(F_p\subseteq \mathbb{Q}_{\infty,p}\). Moreover, the quadratic twist \(E^{(K)}\) possesses a point of order \(p^n\) defined over \(F_p\). Hence \(E^{(K)}(\mathbb{Q}_{\infty,p})\) contains a subgroup isomorphic to \(\mathbb{Z}/p^n\mathbb{Z}\).
    \end{enumerate}
\end{lemma}

\begin{proof}
Set \(F = \mathbb{Q}(P)\). The \(G_{\mathbb{Q}}\)-stability gives a character \(\psi: G_{\mathbb{Q}}\to \Aut(\langle P\rangle)\cong (\mathbb{Z}/p^n\mathbb{Z})^\times\) with \(\overline{\mathbb{Q}}^{\ker\psi}=F\). Thus \(F/\mathbb{Q}\) is abelian of degree dividing \(\varphi(p^n)=p^{n-1}(p-1)\).

Since \(L/K\) is an intermediate Galois field of \(K_{\infty}/K\), it is a finite \(p\)-extension unramified outside \(p\) by Proposition~\ref{prop:allZpext} and Proposition~\ref{prop:unramifiedoutsidep}. Because \(P\in E(L)\) and \(K\subset L\), we have \(F\subset L\) and \(FK\subset L\). Hence \([FK:K]\) is a power of \(p\) and \(FK/K\) is unramified outside \(p\).

We distinguish two cases.

\emph{Case 1: \(K\not\subset F\).} Then \(F\cap K=\mathbb{Q}\) and \([F:\mathbb{Q}] = [FK:K] = p^m\) for some \(m\ge 1\). The field diagram is
\[
\begin{tikzcd}
 & FK \ar[dl, -, "2"'] \ar[dr, -, "p^m"] & \\
F \ar[dr, -, "p^m"'] & & K \ar[dl, -, "2"] \\
 & \mathbb{Q} &
\end{tikzcd}
\]
For any prime \(q\neq p\), denote by \(e_q(\cdot)\) the ramification index.
Since \(FK/K\) is unramified outside \(p\), we have \(e_q(FK/K)=1\) and therefore
\[
e_q(F/\mathbb{Q})\mid e_q(FK/\mathbb{Q}) = e_q(FK/K)\,e_q(K/\mathbb{Q}) = e_q(K/\mathbb{Q})\in\{1,2\}.
\]
On the other hand, since \(F/\mathbb{Q}\) is a \(p\)-extension, \(e_q(F/\mathbb{Q})\) is a power of \(p\). Since \(p\) is odd, this forces
\(e_q(F/\mathbb{Q})=1\).  Hence \(F/\mathbb{Q}\) is an abelian \(p\)-extension
unramified outside \(p\).

By the Kronecker--Weber theorem (see~\cite[Theorem 14.1]{Washington1997}), \(F\subseteq \mathbb{Q}(\zeta_{p^\infty})\). Since \(p\) is odd, \(\Gal{\mathbb{Q}(\zeta_{p^\infty})}{\mathbb{Q}} \cong \Delta \times \mathbb{Z}_p\) with \(\Delta\) of order \(p-1\). So \(F\) lies in the unique \(\mathbb{Z}_p\)-extension \(\mathbb{Q}_{\infty,p}=\mathbb{Q}(\zeta_{p^\infty})^\Delta\). Thus \(P\in E(\mathbb{Q}_{\infty,p})\).

\emph{Case 2: \(K\subset F\).} Then \([F:K]=[FK:K]=p^k\) with \(k\ge 1\), otherwise \(F=K\), contradicting \(P\notin E(K)\). Hence \([F:\mathbb{Q}]=2p^k\). As \(p\) is odd, \(\Gal{F}{\mathbb{Q}}\cong C_2 \times C_{p^k}\) since \(F/\mathbb{Q}\) is abelian. Let \(F_p\) be the fixed field of the \(C_2\)-factor, then
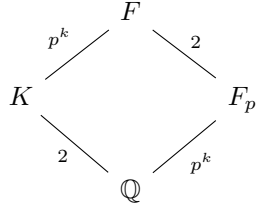
\[
\begin{tikzcd}
 & F \ar[dl,  -,"p^k"'] \ar[dr, -, "2"] & \\
K \ar[dr,  -,"2"'] & & F_p \ar[dl,  -,"p^k"] \\
 & \mathbb{Q} &
\end{tikzcd}
\]
so that \(F=KF_p\), \(K\cap F_p=\mathbb{Q}\), and \([F_p:\mathbb{Q}]=p^k\).

The extension \(F/K\) is unramified outside \(p\) and is isomorphic to
\(F_p/\mathbb{Q}\) via restriction.  Let \(q\neq p\).  Because
\(e_q(F/K)=1\) and \(e_q(K/\mathbb{Q})\in\{1,2\}\), we have
\(e_q(F/\mathbb{Q})\in\{1,2\}\).   As \(p\) is odd,
\(e_q(F_p/\mathbb{Q})=p^a \mid e_q(F/\mathbb{Q})\) and hence  must be \(1\).  Therefore \(F_p/\mathbb{Q}\) is
unramified outside \(p\), and Kronecker--Weber gives \(F_p\subseteq
\mathbb{Q}_{\infty,p}\).

It remains to relate \(P\) to the twist. Let \(\phi: E_K \to E^{(K)}_K\) be the isomorphism \((x,y)\mapsto (x,y/\sqrt{d})\) and set \(Q = \phi(P)\). For any \(\sigma\in G_{\mathbb{Q}}\),
\[
Q^\sigma = \chi_K(\sigma)\psi(\sigma)\, Q,
\]
so \(Q\) is fixed by \(\ker(\chi_K\psi)\). We claim that \(F_p\) lies in the fixed field of \(\chi_K\psi\). Any \(\sigma\in \Gal{\overline{\mathbb{Q}}}{F_p}\) acts on \(F\) through \(\Gal{F}{F_p}\cong C_2\). If the action is trivial, there is nothing to discuss. If this action is non-trivial, then \(\sigma\) acts non-trivially on \(K\), hence \(\chi_K(\sigma)=-1\). Because \(\psi\) is faithful on \(\Gal{F}{F_p}\), we also have \(\psi(\sigma)=-1\). Thus \(\chi_K\psi(\sigma)=1\). Therefore \(Q\in E^{(K)}(F_p)\). As \(Q\) has order \(p^n\), the group \(E^{(K)}(\mathbb{Q}_{\infty,p})\) contains a point of order \(p^n\).
\end{proof}

\begin{lemma}\label{lemma:ell-twist}
    Let \(p\) be an odd prime, \(E/\mathbb{Q}\) an elliptic curve, and \(K\) a quadratic field. 
    Let \(K_{\infty}/K\) be the compositum of all \(\mathbb{Z}_p\)-extensions of \(K\), and let \(L/K\) be an intermediate field with \(L/\mathbb{Q}\) Galois. 
    Let \(q\neq p\) be another odd prime. Suppose that \(P\in E(L)\) is a torsion point of order \(q^n\) (\(n\ge 1\)) with \(P\notin E(K)\), and that \(\langle P\rangle\) is \(G_{\mathbb{Q}}\)-stable.
    Then the following hold:
    \begin{enumerate}
        \item \(p\) divides \(q-1\).
        \item Let \(F=\mathbb{Q}(P)\). Then \([FK:K]=p^{m}\) for some \(m\ge 1\).
        \item If \(K\not\subset F\), then \(F/\mathbb{Q}\) is an abelian \(p\)-extension unramified outside \(p\) and \(F\subseteq \mathbb{Q}_{\infty,p}\). In particular \(P\in E(\mathbb{Q}_{\infty,p})\).
        \item If \(K\subset F\), then \(F\) contains a unique subfield \(F_p\) of \(\mathbb{Q}\), such that \([F_p:\mathbb{Q}]=p^m\), unramified outside \(p\), with \(F_p\subseteq \mathbb{Q}_{\infty,p}\). Moreover, the quadratic twist \(E^{(K)}\) has a point of order \(q^n\) defined over \(F_p\). Hence \(E^{(K)}(\mathbb{Q}_{\infty,p})\) contains \(\mathbb{Z}/q^n\mathbb{Z}\).
    \end{enumerate}
\end{lemma}

\begin{proof}
Let \(\psi: G_{\mathbb{Q}}\to \Aut(\langle P\rangle)\cong (\mathbb{Z}/q^n\mathbb{Z})^\times\) be the character coming from the Galois action, and set \(F=\mathbb{Q}(P)\). Then \(F/\mathbb{Q}\) is abelian of degree dividing \(\varphi(q^n)=q^{n-1}(q-1)\).

As in the proof of Lemma~\ref{lemma:p-twist}, \(L/K\) is a finite \(p\)-extension unramified outside \(p\), and \(F\subset L\), \(FK\subset L\) imply \([FK:K]=p^m\) with \(m\ge 0\). The hypothesis \(P\notin E(K)\) forces \(m\ge 1\) (otherwise \(F\subset K\)). Hence \([F:\mathbb{Q}] = p^m \cdot d\) where \(d=[F\cap K:\mathbb{Q}]\in\{1,2\}\). Since this degree divides \(q^{n-1}(q-1)\) and \(q\neq p\), we must have \(p\mid(q-1)\). This proves (1) and (2).

We now distinguish the two possibilities for \(d\).

\emph{Case 1: \(K\not\subset F\) (\(d=1\)).} Then \(F\cap K=\mathbb{Q}\), \([F:\mathbb{Q}]=p^m\), and the field diagram is identical to Case~1 of Lemma~\ref{lemma:p-twist}. The ramification argument shows that \(F/\mathbb{Q}\) is unramified outside \(p\). Hence \(F\) is an abelian \(p\)-extension, and Kronecker--Weber gives \(F\subseteq \mathbb{Q}_{\infty,p}\). This yields (3).

\emph{Case 2: \(K\subset F\) (\(d=2\)).} Then \([F:\mathbb{Q}]=2p^m\) and \([F:K]=p^m\). As \(p\) is odd, \(\Gal{F}{\mathbb{Q}}\cong C_2 \times C_{p^m}\). Let \(F_p\) be the fixed field of the \(C_2\)-factor. then \(F_p/\mathbb{Q}\) is a \(p\)-extension of degree \(p^m\), \(F=KF_p\), and \(K\cap F_p=\mathbb{Q}\), exactly as in Case~2 of Lemma~\ref{lemma:p-twist}. The same ramification argument shows that \(F_p\) is unramified outside \(p\), and Kronecker--Weber yields \(F_p\subseteq \mathbb{Q}_{\infty,p}\).

Finally, set \(Q=\phi(P)\in E^{(K)}\) where \(\phi(x,y)=(x,y/\sqrt{d})\). For \(\sigma\in G_{\mathbb{Q}}\),
\[
Q^\sigma = \chi_K(\sigma)\psi(\sigma)\,Q.
\]
By the same faithful character argument as before, \(\Gal{\overline{\mathbb{Q}}}{F_p}\) fixes \(Q\). Thus \(Q\in E^{(K)}(F_p)\). The order of \(Q\) is \(q^n\), completing the proof of (4).
\end{proof}

\subsection{Known results on torsion subgroups}\label{subsection:knownresults}
In this subsection, we collect some known results on torsion subgroups, which will be used to prove the main theorem in this article.

In this subsection, let \(E\) be an elliptic curve over \(\mathbb{Q}\). First is Mazur's theorem mentioned in the Introduction.

\begin{theorem}[\cite{mazurModularCurvesEisenstein1977}, Theorem (8)]\label{theorem:mazurQ}
    Then
\[
E(\mathbb{Q})_{\mathrm{tors}} \simeq
\begin{cases}
\mathbb{Z}/N\mathbb{Z} & \text{with } 1 \leq N \leq 10 \text{ or } N = 12, \\
\mathbb{Z}/2\mathbb{Z} \times \mathbb{Z}/2N\mathbb{Z} & \text{with } 1 \leq N \leq 4.
\end{cases}
\]
\end{theorem}

The following theorem is about torsion subgroups of rational elliptic curves over quadratic, quartic and sextic fields.

\begin{theorem}[\cite{Najman2016}, Theorem 2]\label{theorem:NajmanK}
Let \( E/\mathbb{Q} \) be an elliptic curve, and let \( K \) be a quadratic number field. Then
\[
E(K)_{\mathrm{tors}} \simeq
\begin{cases}
\mathbb{Z}/N\mathbb{Z} & \text{with } 1 \leq N \leq 10 \text{ or } N = 12, 15, 16, \\
\mathbb{Z}/2\mathbb{Z} \times \mathbb{Z}/2N\mathbb{Z} & \text{with } 1 \leq N \leq 6, \\
\mathbb{Z}/3\mathbb{Z} \times \mathbb{Z}/3N\mathbb{Z} & \text{with } N = 1, 2, \text{ only if } K = \mathbb{Q}(\sqrt{-3}), \\
\mathbb{Z}/4\mathbb{Z} \times \mathbb{Z}/4\mathbb{Z} & \text{only if } K = \mathbb{Q}(\sqrt{-1}).
\end{cases}
\]
\end{theorem}

\begin{theorem}[\cite{chou2016}, Theorem 1.2]\label{theorem:chou-quartic-galois}
Let \( E/\mathbb{Q} \) be an elliptic curve, and let \( K \) be a quartic Galois extension of \( \mathbb{Q} \).  
Then \( E(K)_{\mathrm{tors}} \) is isomorphic to one of the following groups:

\[
\begin{aligned}
&\mathbb{Z}/N_1\mathbb{Z}, \quad N_1 = 1, \dots, 16, \; N_1 \neq 11, 14,\\
&\mathbb{Z}/2\mathbb{Z} \oplus \mathbb{Z}/2N_2\mathbb{Z}, \quad N_2 = 1, \dots, 6, 8,\\
&\mathbb{Z}/3\mathbb{Z} \oplus \mathbb{Z}/3N_3\mathbb{Z}, \quad N_3 = 1, 2,\\
&\mathbb{Z}/4\mathbb{Z} \oplus \mathbb{Z}/4N_4\mathbb{Z}, \quad N_4 = 1, 2,\\
&\mathbb{Z}/5\mathbb{Z} \oplus \mathbb{Z}/5\mathbb{Z},\\
&\mathbb{Z}/6\mathbb{Z} \oplus \mathbb{Z}/6\mathbb{Z}.
\end{aligned}
\]
\end{theorem}

\begin{theorem}[\cite{adžaga2026completingclassificationtorsionsubgroups}, Theorem 1.2]\label{theorem:ag-six-field}
    Let \( E/\mathbb{Q} \) be an elliptic curve, and let \( K \) be a sextic number field.  
Then \( E(K)_{\mathrm{tors}} \) is isomorphic to one of the following groups:

\[
\begin{aligned}
&\mathbb{Z}/N_1\mathbb{Z}, \quad N_1 = 1, \dots, 16, 18, 21, 30, \; N_1 \neq 11,\\
&\mathbb{Z}/2\mathbb{Z} \oplus \mathbb{Z}/2N_2\mathbb{Z}, \quad N_2 = 1, \dots, 7, 9,\\
&\mathbb{Z}/3\mathbb{Z} \oplus \mathbb{Z}/3N_3\mathbb{Z}, \quad N_3 = 1, 2, 3, 4,\\
&\mathbb{Z}/4\mathbb{Z} \oplus \mathbb{Z}/4N_4\mathbb{Z}, \quad N_4 = 1, 3,\\
&\mathbb{Z}/6\mathbb{Z} \oplus \mathbb{Z}/6\mathbb{Z}.
\end{aligned}
\]
\end{theorem}

Now we introduce theorems about torsion subgroups of rational elliptic curves over infinite extensions.

\begin{theorem}[\cite{chou2021},Theorem 1.1, 1.2, 1.3]\label{theorem:chouQinfty}
Let \(E/\mathbb{Q}\) be an elliptic curve, \(p\) a prime number, and let \(\mathbb{Q}_{\infty,p}\) denote the cyclotomic \(\mathbb{Z}_p\)-extension of \(\mathbb{Q}\). Then:

\begin{enumerate}
\item[(i)] If \(p \ge 5\), then
\[
E(\mathbb{Q}_{\infty,p})_{\mathrm{tors}} = E(\mathbb{Q})_{\mathrm{tors}}.
\]

\item[(ii)] If \(p=2\), then \(E(\mathbb{Q}_{\infty,2})_{\mathrm{tors}}\) is isomorphic to one of the following groups:
\[
\mathbb{Z}/N\mathbb{Z} \quad (1\le N\le 10,\ \text{or } N=12),
\]
or
\[
\mathbb{Z}/2\mathbb{Z} \oplus \mathbb{Z}/2N\mathbb{Z} \quad (1\le N\le 4).
\]
Moreover, for each group \(G\) in the above list, there exists an elliptic curve \(E/\mathbb{Q}\) such that \(E(\mathbb{Q}_{\infty,2})_{\mathrm{tors}} \simeq G\).

\item[(iii)] If \(p=3\), then \(E(\mathbb{Q}_{\infty,3})_{\mathrm{tors}}\) is isomorphic to one of the following groups:
\[
\mathbb{Z}/N\mathbb{Z} \quad (1\le N\le 10,\ \text{or } N=12,21,27),
\]
or
\[
\mathbb{Z}/2\mathbb{Z} \oplus \mathbb{Z}/2N\mathbb{Z} \quad (1\le N\le 4).
\]
Moreover, for each group \(G\) in the above list, there exists an elliptic curve \(E/\mathbb{Q}\) such that \(E(\mathbb{Q}_{\infty,3})_{\mathrm{tors}} \simeq G\).
\end{enumerate}
\end{theorem}

Moreover, we quote the result of Avc{\i} which we will generalize.

\begin{theorem}[\cite{AVCI2026153}, Theorem 1, Theorem 5.1, 5.2, 5.3]\label{theorem:Avci}
Let \(E/\mathbb{Q}\) be an elliptic curve, let \(K\) be a quadratic field, \(p\) a prime, and \(L/K\) a \(\mathbb{Z}_p\)-extension. Then

\begin{enumerate}
\item[(1)] If \(p>5\), then
\[
E(L)_{\mathrm{tors}} = E(K)_{\mathrm{tors}}.
\]

\item[(2)] If \(p=5\), then either
\[E(L)_{\mathrm{tors}} \cong \mathbb{Z}/25\mathbb{Z} \quad
\text{or} \quad E(L)_{\mathrm{tors}} = E(K)_{\mathrm{tors}}\]

\item[(3)] If \(p=3\), we distinguish the type of extension:
\begin{itemize}
\item If \(L=K_{cyc}\) is the cyclotomic \(\mathbb{Z}_3\)-extension, then \(E(K_{cyc})_{\mathrm{tors}}\) is isomorphic to one of the following groups:
\[
\begin{cases}
\mathbb{Z}/N\mathbb{Z} & \text{with } 1\le N\le 10,\ \text{or } N=12,13,15,16,18,21,27,\\[2mm]
\mathbb{Z}/2\mathbb{Z}\times \mathbb{Z}/2N\mathbb{Z} & \text{with } 1\le N\le 7,\ \text{or } N=9,\\[2mm]
\mathbb{Z}/3\mathbb{Z}\times \mathbb{Z}/3N\mathbb{Z} & \text{with } N=1,2,3,\ \text{only if } K=\mathbb{Q}(\sqrt{-3}),\\[2mm]
\mathbb{Z}/4\mathbb{Z}\times \mathbb{Z}/4\mathbb{Z} & \text{only if } K=\mathbb{Q}(\sqrt{-1}).
\end{cases}
\]

\item If \(K=\mathbb{Q}(\sqrt{-d})\) is imaginary quadratic with \(d\) a square-free positive integer and \(d\ne 1,3\), and \(L=K_{anti}\) is the anticyclotomic \(\mathbb{Z}_3\)-extension, then \(E(K_{anti})_{\mathrm{tors}}\) is isomorphic to one of the following groups:
\[
\begin{cases}
\mathbb{Z}/N\mathbb{Z} & \text{with } 1\le N\le 10,\ \text{or } N=12,15,16,\\[2mm]
\mathbb{Z}/2\mathbb{Z}\times \mathbb{Z}/2N\mathbb{Z} & \text{with } 1\le N\le 7,\ \text{or } N=9.
\end{cases}
\]
\end{itemize}
\end{enumerate}
\end{theorem}

\section{The case: \(p \ge 5\)}\label{section:p=5}

In this section we prove the case \(p \ge 5\) of the main theorem.

\begin{theorem}[\(p \ge 5\)]\label{theorem:p>=5}
    Let \(E/\mathbb{Q}\) be a rational elliptic curve, let \(K\) be a quadratic field, and let \(K_\infty/K\) be the compositum of all \(\mathbb{Z}_p\)-extensions of \(K\) for a prime \(p \ge 5\). Then
    \[
    E(K)_{\mathrm{tors}} = E(K_\infty)_{\mathrm{tors}}.
    \]
\end{theorem}

\begin{proof}
    In view of Theorem~\ref{theorem:Avci}(1), we must address two issues:
    \begin{itemize}
        \item extending the result from a single \(\mathbb{Z}_p\)-extension to the compositum of all \(\mathbb{Z}_p\)-extensions;
        \item ruling out the exceptional case \(\mathbb{Z}/25\mathbb{Z}\) when \(p = 5\).
    \end{itemize}

    We first address a subtle point in the proof of \cite{AVCI2026153} that is relevant for passing from a single \(\mathbb{Z}_p\)-extension to the compositum. In \cite[Theorem~1.1]{AVCI2026153}, Avc{\i}  proves that for a fixed \(\mathbb{Z}_p\)-extension \(L/K\) with \(p > 5\), one has \(E(L)_{\mathrm{tors}} = E(K)_{\mathrm{tors}}\). The argument there uses lemmas, specifically Lemmas 4.4 and 4.5 of \cite{AVCI2026153} that require the finite intermediate fields to be Galois over \(\mathbb{Q}\). However, for an arbitrary \(\mathbb{Z}_p\)-extension \(L/K\), especially when \(K\) is imaginary quadratic, its finite subfields \(F\) with \(K \subseteq F \subseteq L\) need not be Galois over \(\mathbb{Q}\). 

    Nevertheless, this gap is easily fixed by replacing \(F\) with its Galois closure \(F'\) over \(\mathbb{Q}\). Since \(K/\mathbb{Q}\) is quadratic and \(F/K\) is a cyclic subfield  of \(p\)-power degree in a \(\mathbb{Z}_p\)-extension, all conjugates of \(F\) lie inside the maximal \(\mathbb{Z}_p\)-extension of \(K\), so the compositum \(F'\) satisfies \([F' : K] = p^b\) for some \(b\). Hence \([F' : \mathbb{Q}] = 2p^b\), which is not divisible by \(3\) or \(4\) for \(p \ge 5\). Thus \(F'/\mathbb{Q}\) is a finite Galois extension to which all the lemmas of \cite{AVCI2026153} apply, giving \(E(F')_{\mathrm{tors}}[q^\infty] = E(K)_{\mathrm{tors}}[q^\infty]\) for \(q=2,3\). Since \(K \subseteq F \subseteq F'\), we obtain \(E(F)_{\mathrm{tors}}[q^\infty] = E(K)_{\mathrm{tors}}[q^\infty]\) for \(q=2,3\) as well. As every finite subfield of \(L\) shares this property, the conclusion \(E(L)_{\mathrm{tors}}[q^\infty] = E(K)_{\mathrm{tors}}[q^\infty]\) for \(q=2,3\) follows by taking the union. Therefore, the corrected argument establishes Avc{\i}'s result,  for \(q=2,3\),  for every single \(\mathbb{Z}_p\)-extension, without any further modification of the remaining proof. 

    Now let \(K_\infty\) be the compositum of all \(\mathbb{Z}_p\)-extensions of \(K\).
    Let \(F/K\) be any finite subextension of \(K_\infty/K\) and let \(F'\) be the Galois
    closure of \(F\) over \(\mathbb{Q}\).
    Since \(K_\infty\) is Galois over \(\mathbb{Q}\), \(F'\subseteq K_\infty\).
    Moreover, \(F'/K\) is a finite Galois extension whose Galois group is a quotient of
    \(\Gal{K_\infty}{K}\cong \mathbb{Z}_p^r\), so \([F':K]\) is a power of \(p\), say
    \(p^b\).  Thus \([F':\mathbb{Q}]=2p^b\), and because \(p\ge 5\), this degree is not
    divisible by \(3\) or \(4\).
    
    The lemmas of \cite{AVCI2026153} therefore apply to the Galois
    extension \(F'/\mathbb{Q}\) and give \(E(F')_{\mathrm{tors}}[q^\infty]=E(K)_{\mathrm{tors}}[q^\infty]\) for \(q=2,3\).
    Since \(K\subseteq F\subseteq F'\), we obtain \(E(F)_{\mathrm{tors}}[q^\infty]=E(K)_{\mathrm{tors}}[q^\infty]\) for \(q=2,3\).
    Passing to the direct limit over all such \(F\) yields
    \(E(K_\infty)_{\mathrm{tors}}[q^\infty]=E(K)_{\mathrm{tors}}[q^\infty]\) for \(q=2,3\). 

    For  \(q \ge 5\), by Lemma~4.7 there,  \(\mu_q\) is not contained in any \(\mathbb{Z}_p\)-extension of \(K\). An easy arguement involving ramification index showed that \(K_\infty\) does not contain \(\mu_q\) and hence \(E(K_\infty)[q^\infty] \) is cyclic by Weil pairing. By Lemma~\ref{lemma:KQPabelian}, \(K\mathbb{Q}(P)/\mathbb{Q}\) is abelian, where \(P\) is a point of order \(q^k\) in \(E(K_\infty)[q^\infty] \). Since \(p\) is odd , \(K\mathbb{Q}(P)\) is contained in \(K_{cyc}\), which is a single \(\mathbb{Z}_p\)-extension of \(K\). Now apply the  main theorem of Avc{\i} and its argument, for \(p>5\) and other \(q\), this is done for \(K_\infty\).
    
    Thus the first issue is resolved. It remains only to rule out the exceptional \(\mathbb{Z}/25\mathbb{Z}\) torsion in the case \(p = 5\), which we now address.
    
    Next, we exclude the possibility of \(\mathbb{Z}/25\mathbb{Z}\). Assume, for contradiction, that there exists \(P \in E(F)\) of order \(25\) and that \(E(F)_{\mathrm{tors}}[5^\infty] \cong \mathbb{Z}/25\mathbb{Z}\) for some intermediate field \(F/K\) with \(F/\mathbb{Q}\) Galois. The cyclic subgroup \(\langle P \rangle \subset E(F)[25]\) is \(G_{\mathbb{Q}}\)-stable. Thus \(E\) admits a rational \(25\)-isogeny by Lemma~\ref{lemma:admits-n-isogeny}, and \(G_{\mathbb{Q}}\) acts on \(\langle P \rangle\) through a character
    \[
    \psi : G_{\mathbb{Q}} \to \Aut(\langle P \rangle) \cong (\mathbb{Z}/25\mathbb{Z})^\times.
    \]
    Hence \(\mathbb{Q}(P)/\mathbb{Q}\) is abelian with \([\mathbb{Q}(P) : \mathbb{Q}] \mid \varphi(25) = 20\). Since \(\mathbb{Q}(P) \subseteq K_\infty\), we have
    \[
    [\mathbb{Q}(P) : \mathbb{Q}] \in \{1, 2, 5, 10\}.
    \]
    By Theorem~\ref{theorem:mazurQ} and Theorem~\ref{theorem:NajmanK}, \([\mathbb{Q}(P) : \mathbb{Q}]\) cannot be \(1\) or \(2\), otherwise \(E(K)\) or \(E(\mathbb{Q})\) would contain a torsion point of order \(25\).

    Now \(P \notin E(K)\). Lemma~\ref{lemma:p-twist} can be applied in this case. Then either \(E(\mathbb{Q}_{\infty,p})\) contains a torsion point of order \(25\), or \(E^{(K)}(\mathbb{Q}_{\infty,p})\) contains a torsion point of order \(25\) which is not contained in \(E(\mathbb{Q})\) or \(E^{(K)}(\mathbb{Q})\). This contradicts Theorem~\ref{theorem:chouQinfty}.

    Therefore, we complete the proof.
\end{proof}
    \section{The case: \(p=3\)}\label{section:p=3}
In this section, we discuss the case \(p=3\). We first prepare some lemmas. The following lemma bounds the size of the group of roots of unity in a \(\mathbb{Z}_3\)-extension. It is similar to \cite[Lemma 4.7]{AVCI2026153}, which treated \(p>3\).

\begin{lemma}\label{lemma:rootsofunity-p=3}
Let \(K=\mathbb{Q}(\sqrt{d})\) be a quadratic field, where \(d\) is a square-free integer, and let \(L\) be a \(\mathbb{Z}_3\)-extension of \(K\).
\begin{enumerate}
\item If \(K \neq \mathbb{Q}(\sqrt{-3})\), then every root of unity in \(L\) already lies in \(K\), i.e., \(\mu(L)=\mu(K)\).
\item If \(K = \mathbb{Q}(\sqrt{-3})\), then  \(\mu(L)=\mu(L\cap K_{cyc})=\mu(L\cap \mathbb{Q}(\zeta_{3^\infty})) \).
\end{enumerate}
\end{lemma}

\begin{proof}
We follow the strategy of \cite[Lemma 4.7]{AVCI2026153}. Suppose \(\mu_n \subset L\) and set \(F = K(\mu_n) \subset L\). Then \([F:K]\) is a power of \(3\).

\noindent\textbf{Case 1: \(K \neq \mathbb{Q}(\sqrt{-3})\).}
First we show that \(3 \nmid n\). If \(3 \mid n\), then \(\zeta_3 \in F\), hence \(K(\zeta_3) = K(\sqrt{-3}) \subset F\). Because \(K \neq \mathbb{Q}(\sqrt{-3})\), the intersection \(K \cap \mathbb{Q}(\sqrt{-3}) = \mathbb{Q}\), so \([K(\zeta_3):K] = 2\). This contradicts the fact that \([F:K]\) is a power of \(3\). Thus \(3 \nmid n\).

Now suppose a prime \(q>3\) divides \(n\). Then \(q \neq 3\). Every \(\mathbb{Z}_3\)-extension is unramified outside \(3\), so \(q\) is unramified in \(L/K\). However, \(K(\zeta_q) \subset L\) and \(\mathbb{Q}(\zeta_q)/\mathbb{Q}\) is totally ramified at \(q\) with ramification index \(q-1\). The ramification index of \(K/\mathbb{Q}\) at \(q\) is at most \(2\). Therefore the ramification index of \(K(\zeta_q)/K\) at \(q\) is at least \((q-1)/2 > 1\). Hence \(q\) ramifies in \(L/K\), contradiction. Consequently, \(n\) has no prime divisor larger than \(3\).

Combining the two steps, the only possible prime divisor of \(n\) is \(2\). Thus \(n=2^a\), and \(\varphi(n)\) is a power of \(2\). Since
\[
[F:K] = \frac{[\mathbb{Q}(\mu_n):\mathbb{Q}]}{[K \cap \mathbb{Q}(\mu_n):\mathbb{Q}]} = \varphi(n) \quad\text{or}\quad \frac{\varphi(n)}{2},
\]
\([F:K]\) is also a power of \(2\). But \([F:K]\) is a power of \(3\), therefore \([F:K]=1\) and \(F=K\). Hence \(\mu_n \subset K\). Since \(n\) was arbitrary, \(\mu(L) = \mu(K)\).

\noindent\textbf{Case 2: \(K = \mathbb{Q}(\sqrt{-3})\).} Suppose that \(\mu_n \subset L\).
 Exactly as in Case~1, \(q \nmid n\)  for any prime \(q>3\). Thus,  the only prime divisors of \(n\) are \(2\) and \(3\), and we first show the \(2\)-part cannot be large.

If \(4 \mid n\), then \(\zeta_4 = i \in L\). Consider \(K(i) = \mathbb{Q}(\zeta_3,i) = \mathbb{Q}(\zeta_{12})\). We have \([\mathbb{Q}(\zeta_{12}):\mathbb{Q}]=4\) and \([K:\mathbb{Q}]=2\), so \([K(i):K]=2\). But \(K(i) \subset L\) forces \([K(i):K]\) to be a power of \(3\), impossible. Similarly, higher powers of \(2\) are excluded. Hence the \(2\)-part of \(n\) is at most \(2\).

Note that \(\mu(L\cap K_{cyc}) \subseteq \mu(L)\) is clear, and we need to show \(\mu(L) \subseteq \mu(L\cap K_{cyc})\) or equivalently, \(\mu(L) \subseteq \mu(K_{cyc})\). By the discussion above, we know that \(n=2\cdot 3^b\). Hence \(\mu(L) \subset \mu(\mathbb{Q}(\zeta_{3^\infty}))\), and the latter is exactly \(K_{cyc}\). Hence we have proved this lemma.
\end{proof}

Note that \(\Gal{K_\infty}{K}\cong \mathbb{Z}_3^2\),  then, by Galois theory, every \(\mathbb{Z}_3\)-extension corresponds to a closed subgroup of \(\Gal{K_\infty}{K}\). Write \(\Gal{K_\infty}{K}= \mathbb{Z}_3e_1\oplus \mathbb{Z}_3e_2\), where \(\mathbb{Z}_3e_1\) corresponds to the cyclotomic \(\mathbb{Z}_3\)-extension and \(\mathbb{Z}_3e_2\) corresponds to the anticyclotomic \(\mathbb{Z}_3\)-extension, as closed  subgroup. Then any \(\mathbb{Z}_3\)-extension corresponds to the closed subgroup generated by \(ae_1+be_2\) with \(a,b\in \mathbb{Z}_3\) and not both \(a\) and \(b\) divisible by \(3\). Suppose that \(L\) corresponds to the closed subgroup \(\mathbb{Z}_3(ae_1+be_2)\), then \(L\cap K_{cyc}=K_{cyc}^{v_3(b)}\), where \(v_3(b)\) is the \(3\)-adic valuation of \(b\).

Therefore if \(b\) is a \(3\)-adic unit, then \(\mu(L)=\mu(K)\).

\begin{proposition}\label{prop:reducetocyc}
    Let \(E\) be an elliptic curve over \(\mathbb{Q}\) and \(K\) be a quadratic field, with \(K_\infty\) the compositum of all \(\mathbb{Z}_3\)-extensions of \(K\). We have
    \[
    E(K_\infty)_{\mathrm{tors}}[q^\infty]=E(K_{cyc})_{\mathrm{tors}}[q^\infty]
    \]
    for all odd primes \(q\), except for the case \(K=\mathbb{Q}(\sqrt{-3})\) with \(q=3\).
\end{proposition}

\begin{proof}
Suppose we are not in the exceptional case \(K=\mathbb{Q}(\sqrt{-3})\) and \(q=3\). By Lemma~\ref{lemma:rootsofunity-p=3}, \(K_\infty\) does not contain \(\mu_q\) since \(K(\zeta_q)\subseteq K_{cyc}\), and by Lemma~\ref{lemma:rootsofunity}, we have \(E(K_\infty)[q] \neq E[q]\). Hence \(E(K_\infty)[q]\) is either trivial or isomorphic to \(\mathbb{Z}/q\mathbb{Z}\).

Therefore, for any point \(P\in E(K_\infty)\) of order \(q^a\), the cyclic subgroup generated by \(P\) is \(G_\mathbb{Q}\)-stable. By Lemma~\ref{lemma:KQPabelian}, it follows that \(K\mathbb{Q}(P)/\mathbb{Q}\) is abelian.

Since every abelian subextension containing \(K\) of \(K_\infty/\mathbb{Q}\) is contained in \(K_{cyc}\), we obtain \(K\mathbb{Q}(P)\subseteq K_{cyc}\), which implies \(P\in E(K_{cyc})\). 
\end{proof}

\begin{remark}
    So far, no results from Section~\ref{subsection:knownresults} have been used. In particular, Proposition~\ref{prop:reducetocyc} does not depend on any detailed classification of torsion subgroups of \(E\) over arbitrary fields.
\end{remark}

According to Lemma~\ref{lemma:rootsofunity-p=3} and Proposition~\ref{prop:reducetocyc}, the cases are very different according as \(K \neq \mathbb{Q}(\sqrt{-3})\) or \(K= \mathbb{Q}(\sqrt{-3})\). We now discuss the two possibilities for \(K\) separately and use results in Section~\ref{subsection:knownresults} to obtain refined results.

Before the discussion, we quote a lemma which will be used in both cases.

\begin{lemma}[\cite{avci2026actaarith}, Lemma 3.1]\label{lemma:avci-lemma3.1}
Let \(E/\mathbb{Q}\) be an elliptic curve, and let \(F\) be a Galois number field of degree \(n\). Let \(q^k\) be an odd prime power. Suppose that
\[
E(F)[q^k] \cong \mathbb{Z}/q^k\mathbb{Z}.
\]
If \(\gcd(n, \varphi(q^k)) = 1\), then
\[
E(F)[q^k] = E(\mathbb{Q})[q^k].
\]
If \(\gcd(n, \varphi(q^k)) = 2\), then there exists a square-free integer \(d\) such that \(\sqrt{d} \in F\) and
\[
E(F)[q^k] = E(\mathbb{Q}(\sqrt{d}))[q^k].
\]
\end{lemma}

\subsection{The case \(K \neq \mathbb{Q}(\sqrt{-3})\)}

Let \(F/K\) be an intermediate field in \(K_\infty\) with \(F/\mathbb{Q}\) Galois. Suppose that 
\[
E(F)_{\mathrm{tors}}\cong \mathbb{Z}/n\mathbb{Z} \times \mathbb{Z}/mn\mathbb{Z}.
\]
Then by Lemma~\ref{lemma:rootsofunity} we have \(\mu_n \subset F\), and by Lemma~\ref{lemma:admits-n-isogeny}, \(E\) admits an isogeny of degree \(m\). 

Moreover, by Theorem~\ref{theorem:isogeniesoverQ}, \(m\le 19\) or \(m\in\{21,25,27,37,43,67,163 \}\). Hence the possibilities for \(m\) are finite.

In this case, it suffices to show that, for each prime \(q\) that can appear in the degree \(m\) of an isogeny,
\[
E(F)_{\mathrm{tors}}[q^\infty] = E(K)_{\mathrm{tors}}[q^\infty]
\]
for all \(F/K\) inside \(K_\infty\) with \(F/\mathbb{Q}\) Galois.  The relevant primes are
\[
q \in \{2,3,5,7,11,13,17,19,37,43,67,163\}.
\]

The following is the main theorem in this subsection. Recall that by Proposition~\ref{prop:reducetocyc}, we have \(E(K_\infty)_{\mathrm{tors}}[q^\infty]=E(K_{cyc})_{\mathrm{tors}}[q^\infty]\), the following theorem is finer than the proposition.

\begin{theorem}\label{theorem:p=3not-3}
    Let \(E/\mathbb{Q}\) be a rational elliptic curve and let \(K\) be a quadratic field not equal to \(\mathbb{Q}(\sqrt{-3})\). Denote by \(K_{\infty}\) the compositum of all \(\mathbb{Z}_3\)-extensions of \(K\).
\begin{itemize}
    \item For any prime \(q>7\) or \(q=5\),
    \[
    E(K_{\infty})_{\mathrm{tors}}[q^{\infty}] = E(K)_{\mathrm{tors}}[q^{\infty}].
    \]
    \item For \(q=7\),
    \[
    E(K_\infty)_{\mathrm{tors}}[7^\infty]=E(K_{cyc}^1)_{\mathrm{tors}}[7^\infty].
    \]
    Here, \(K_{cyc}^1\) is the first layer of the cyclotomic \(\mathbb{Z}_3\)-extension.

    \item For \(q=3\),
    \[
    E(K_\infty)_{\mathrm{tors}}[3^\infty]=E(K_{cyc})_{\mathrm{tors}}[3^\infty].
    \]
    Moreover, we have
    \begin{itemize}
         \item If \(E(K)_{\mathrm{tors}}[3^\infty]\) is trivial, then \(E(K_\infty)_{\mathrm{tors}}[3^\infty]\) is trivial.

         \item If \(E(K)_{\mathrm{tors}}[3^\infty] \cong \mathbb{Z}/9 \mathbb{Z}\), then \(E(K_\infty)_{\mathrm{tors}}[3^\infty] \cong \mathbb{Z}/9\mathbb{Z}\).

         \item If \(E(K)_{\mathrm{tors}}[3^\infty] \cong \mathbb{Z}/3 \mathbb{Z}\):
         
            \begin{itemize}
                \item If \(E\) or \(E^{(K)}\) is an elliptic curve with Cremona label \(\mathrm{27a4}\), then \(E(K_{cyc}^n)_{\mathrm{tors}}[3^\infty] \cong \mathbb{Z}/3^{n+1} \mathbb{Z}\) for \(n=1,2\) and \(E(K_{cyc}^n)_{\mathrm{tors}}[3^\infty] \cong \mathbb{Z}/27 \mathbb{Z}\) for \(n>2\). Hence in this case \(E(K_\infty)_\mathrm{tors}[3^\infty]\cong \mathbb{Z}/27\mathbb{Z}\).

                \item If  \(E\ncong \mathrm{27a4}\)  and  \(E^{(K)}\ncong \mathrm{27a4}\) , then either \(E(K_\infty)_\mathrm{tors}[3^\infty]=E(K)_{\mathrm{tors}}[3^\infty] \cong \mathbb{Z}/3 \mathbb{Z}\) or \(E(K_\infty)_\mathrm{tors}[3^\infty]\cong  \mathbb{Z}/9 \mathbb{Z}\), with growth occurring in \( E(K_{cyc}^1)_\mathrm{tors}[3^\infty] \cong \mathbb{Z}/9\mathbb{Z}\). A criterion for whether the growth occurs, using Kummer theory and Galois cohomology, is given by Proposition~\ref{prop:3-kummertheory}.
            \end{itemize}
     \end{itemize}
    
    \item For \(q=2\),
    \[E(K_\infty)_\mathrm{tors}[2^\infty]=E(K_\infty^{(1)})_\mathrm{tors}[2^\infty]\]
    Here, \(K_\infty^{(1)}\) is the compositum of all first layers of \(\mathbb{Z}_3\)-extensions. Details about the \(2\)-primary part are described by Proposition~\ref{prop:2-torsion-KneqQsqrt-3}.
\end{itemize}

\end{theorem}

We first introduce some propositions.

\begin{proposition}\label{prop:growth-pattern-q>3}
Let \(q >3\) be a prime and let \(F\) be a subfield in \(K_{cyc}/K\). Then \(E(F)_{\mathrm{tors}}[q^\infty]\) is either trivial or a cyclic group of \(q\)-power order. Moreover, if growth occurs, i.e.\ \(E(K)_{\mathrm{tors}}[q^\infty] \subsetneq E(F)_{\mathrm{tors}}[q^\infty]\), then necessarily \(E(K)_{\mathrm{tors}}[q^\infty]=0\) and the growth happens in a single step. In this case, every point \(P\in E(F)[q^\infty]\setminus E(K)\) has \(\langle P\rangle\) stable under \(G_{\mathbb{Q}}\).
\end{proposition}

\begin{proof}
If \(E(F)[q]=(\mathbb{Z}/q\mathbb{Z})^2\), then the Weil pairing would force \(\mu_q \subset F\), contradicting Lemma~\ref{lemma:rootsofunity-p=3}. Hence \(E(F)[q]\) cannot contain a subgroup isomorphic to \((\mathbb{Z}/q\mathbb{Z})^2\). Therefore the \(q\)-primary part is cyclic.

By Lemma~\ref{lemma:growthpattern=ellnotp} the growth of \(E(F)_{\mathrm{tors}}[q^\infty]\) proceeds block by block. Because the group is cyclic, there is at most one new block after the base field \(K\). Thus if \(E(K)_{\mathrm{tors}}[q^\infty]\) is non-trivial, no further growth is possible, and the equality holds. If it grows, we must be in the situation \(E(K)_{\mathrm{tors}}[q^\infty]=0\) while \(E(F)_{\mathrm{tors}}[q^\infty] \cong \mathbb{Z}/q^k\mathbb{Z}\) for some \(k\ge 1\).

The new points are then Galois stable because the whole \(q\)-primary part is a cyclic \(G_{\mathbb{Q}}\)-module.  
\end{proof}

Hence Lemma~\ref{lemma:p-twist} and Lemma~\ref{lemma:ell-twist} can be applied to any generator of \(E(F)_{\mathrm{tors}}[q^\infty]\).

We now examine the possible primes \(q\) one by one. 

\subsubsection*{The primes \(q=5,11,17\)}

\begin{proposition}\label{prop:q-equal-5-11-17}
For \(q \in \{5,11,17\}\) we have
\[
E(K_\infty)_{\mathrm{tors}}[q^\infty] = E(K)_{\mathrm{tors}}[q^\infty].
\]
\end{proposition}

\begin{proof}
Since \([F:\mathbb{Q}] = 2\cdot 3^s\), we have \(\gcd(\varphi(q^k),[F:\mathbb{Q}]) = 2\) for every \(k\ge 1\).  If the \(q\)-primary part grew, Proposition~\ref{prop:growth-pattern-q>3} would give a cyclic subgroup \(\mathbb{Z}/q^k\mathbb{Z}\) stable under \(G_{\mathbb{Q}}\).  Lemma~\ref{lemma:avci-lemma3.1} then forces this subgroup to be defined over a quadratic field containing \(K\), i.e.\ over \(K\) itself, contradicting the assumption that growth occurred.  Hence no growth is possible and the equality follows.
\end{proof}

\subsubsection*{The primes \(q=13,19,37,43,67,163\)}

\begin{proposition}\label{prop:q-equal-13-etc}
For \(q \in \{13,19,37,43,67,163\}\) we have
\[
E(K_\infty)_{\mathrm{tors}}[q^\infty] = E(K)_{\mathrm{tors}}[q^\infty].
\]
\end{proposition}

\begin{proof}
Assume that growth occurs.  By Proposition~\ref{prop:growth-pattern-q>3} there exists a point \(P\in E(F)\) of order \(q\) (if the new part has order \(q^k\), it contains a point of order \(q\)) with \(\langle P\rangle\) stable under \(G_{\mathbb{Q}}\) and \(P\notin E(K)\).  Lemma~\ref{lemma:ell-twist} then yields either \(P\in E(\mathbb{Q}_{\infty,3})\) or the quadratic twist \(E^{(K)}\) possesses a point of order \(q\) in \(\mathbb{Q}_{\infty,3}\).  Theorem~\ref{theorem:chouQinfty}(iii) states that for these primes neither \(E(\mathbb{Q}_{\infty,3})\) nor any quadratic twist can contain a point of order \(q\).  This contradiction shows that growth cannot happen.
\end{proof}

\subsubsection*{The prime \(q=7\)}

Unlike the previous primes, the full equality \(E(K_\infty)_{\mathrm{tors}}[7^\infty] = E(K)_{\mathrm{tors}}[7^\infty]\) does \emph{not} hold in general since the \(E(K_{cyc})_{\mathrm{tors}}[7^\infty] = E(K)_{\mathrm{tors}}[7^\infty]\) does not hold in general. In fact, we have the following proposition.

\begin{proposition}
    There exists an elliptic curve \(E/\mathbb{Q}\) such that \(E(\mathbb {Q})_{\mathrm{tors}}[7]=E(K)_{\mathrm{tors}}[7]=0\) for any quadratic field \(K\) and \(E(\mathbb{Q}_{\infty,3})[7]\cong \mathbb{Z}/7\mathbb{Z}\).

    Therefore, for any quadratic field we have 
    \[
    E(K)_{\mathrm{tors}}[7]\subsetneq E(K_{cyc})_{\mathrm{tors}}[7].
    \]
\end{proposition}

\begin{proof}
    By Lemma~6.6 and its proof in \cite[pp.~118--119]{chou2021}, there exists an elliptic curve \(E/\mathbb{Q}\) such that \(E(\mathbb{Q})_{\mathrm{tors}}[7]\) is trivial and \(E(\mathbb {Q}_{\infty,3})_{\mathrm{tors}}[7]\cong \mathbb{Z}/7\mathbb{Z}\). Moreover, the \(7\)-division polynomial of \(E\) has irreducible factors of degrees \(3\) and \(21\). Hence no non-trivial point of \(E[7]\) can be defined over a quadratic field, so \(E(K)_{\mathrm{tors}}[7]\) is trivial for every quadratic field \(K/\mathbb{Q}\). This completes the proof.
\end{proof}

\begin{proposition}\label{prop:p=3,q=7}
    Let \(E/\mathbb{Q}\) be an elliptic curve and \(K\) be a quadratic field. Suppose that \[
    E(K)_{\mathrm{tors}}[7]\subsetneq E(K_{cyc})_{\mathrm{tors}}[7]
    \]
    holds, then
    \[
    E(K_{cyc}^1)_{\mathrm{tors}}[7]= E(K_{cyc})_{\mathrm{tors}}[7]
    \]
    where \(K_{cyc}^1\) is the first layer in \(K_{cyc}\).
\end{proposition}

\begin{proof}
    By the proof of Lemma~\ref{lemma:growthpattern=ellnotp}, if the growth appears, we must have \(d_j=1\). Now taking \(p=3\) and \(q=7\) in Lemma~\ref{lemma:structure}, we have \(d_j=3^{j-1}\) for \(j>1\). Hence \(j=1\), that is, the \(7\)-torsion point must occur in \(K_{cyc}^1\).
\end{proof}

\subsubsection*{The prime \(q=3\)}

Here the growth pattern is no longer given by Lemma~\ref{lemma:growthpattern=ellnotp} because \(q=p\). Note that since \(K\neq \mathbb{Q}(\sqrt{-3})\), Proposition~\ref{prop:reducetocyc} shows that \(3^\infty\)-torsion groups in \(E(K_\infty)\) appear in \(E(K_{cyc})\). We can apply Lemma~\ref{lemma:grothpatternp=ell} and Lemma~\ref{lemma:grothpatternp=ell,depth} to the \(\mathbb{Z}_3\)-extension.

\begin{proposition}\label{prop:Z3structurenonQ-3}
    Let \(E/\mathbb{Q}\) be an elliptic curve and \(K\) be a quadratic field which is not \(\mathbb{Q}(\sqrt{-3})\). Let \(K_{cyc}\) be the cyclotomic \(\mathbb{Z}_3\)-extension of \(K\) with intermediate fields \(K_{cyc}^n\).

    \begin{itemize}
        \item If \(E(K)_{\mathrm{tors}}[3^\infty]\) is trivial, then \(E(K_{cyc})_{\mathrm{tors}}[3^\infty]\) is trivial. Hence \(E(K_{\infty})_{\mathrm{tors}}[3^\infty]\) is trivial.

        \item If \(E(K)_{\mathrm{tors}}[3^\infty]\cong \mathbb{Z}/3^k \mathbb{Z}\) for \(k \ge 1\), then \(E(K_{cyc}^n)_{\mathrm{tors}}[3^\infty]\cong \mathbb{Z}/3^{\min\{k+n,s\}} \mathbb{Z}\), where \(s\) is a positive integer defined by \(E(K_{cyc})_\mathrm{tors}[3^\infty]\cong \mathbb{Z}/3^s \mathbb{Z}\), or \(s=\infty\) if \(E(K_{cyc})_\mathrm{tors}[3^\infty]\cong \mathbb{Q}_3/\mathbb{Z}_3\).
    \end{itemize}
\end{proposition}

\begin{proof}
    This is a direct application of Lemma~\ref{lemma:grothpatternp=ell} and Lemma~\ref{lemma:grothpatternp=ell,depth}.
\end{proof}

Proposition~\ref{prop:Z3structurenonQ-3} is a purely abstract algebraic deduction, obtained without using any concrete known information. In what follows, we will use the known information to refine the result.

\begin{theorem}\label{theorem:growthppterQneqQ-3}
     Let \(E/\mathbb{Q}\) be an elliptic curve and \(K\) be a quadratic field which is not \(\mathbb{Q}(\sqrt{-3})\). Let \(K_{cyc}\) be the cyclotomic \(\mathbb{Z}_3\)-extension of \(K\) with intermediate fields \(K_{cyc}^n\), and let \(K_\infty\) be the compositum of all \(\mathbb{Z}_3\)-extensions of \(K\). Then only the following possibilities occur:

     \begin{itemize}
         \item If \(E(K)_{\mathrm{tors}}[3^\infty]\) is trivial, then \(E(K_\infty)_{\mathrm{tors}}[3^\infty]\) is trivial.

         \item If \(E(K)_{\mathrm{tors}}[3^\infty] \cong \mathbb{Z}/9 \mathbb{Z}\), then \(E(K_\infty)_{\mathrm{tors}}[3^\infty] \cong \mathbb{Z}/9\mathbb{Z}\).

         \item If \(E(K)_{\mathrm{tors}}[3^\infty] \cong \mathbb{Z}/3 \mathbb{Z}\):
         
            \begin{itemize}
                \item If \(E\) or \(E^{(K)}\) is isomorphic to the elliptic curve with Cremona label \(\mathrm{27a4}\), then \(E(K_{cyc}^n)_{\mathrm{tors}}[3^\infty] \cong \mathbb{Z}/3^{n+1} \mathbb{Z}\) for \(n=1,2\) and \(E(K_{cyc}^n)_{\mathrm{tors}}[3^\infty] \cong \mathbb{Z}/27 \mathbb{Z}\) for \(n>2\). Hence in this case \(E(K_\infty)_\mathrm{tors}[3^\infty]\cong \mathbb{Z}/27\mathbb{Z}\).

                \item If  \(E\ncong \mathrm{27a4}\)  and  \(E^{(K)}\ncong \mathrm{27a4}\) , then either \(E(K_\infty)_\mathrm{tors}[3^\infty]=E(K)_{\mathrm{tors}}[3^\infty] \cong \mathbb{Z}/3 \mathbb{Z}\) or \(E(K_\infty)_\mathrm{tors}[3^\infty]\cong \mathbb{Z}/9 \mathbb{Z}\), with the growth appearing in \(E(K_{cyc}^1)_\mathrm{tors}[3^\infty] \cong \mathbb{Z}/9\mathbb{Z}\).
            \end{itemize}
     \end{itemize}
\end{theorem}

\begin{proof}
    Under the assumption of \(K\neq \mathbb{Q}(\sqrt{-3})\), by Proposition~\ref{prop:reducetocyc}, we can only consider the cyclotomic \(\mathbb{Z}_3\)-extension.  
    
    By Najman's theorem, i.e., Theorem~\ref{theorem:NajmanK}, we know that \(E(K)_{\mathrm{tors}}[3^\infty]\) has only three possibilities: the trivial group, a cyclic group of order \(3\), and a cyclic group of order \(9\).

    If \(E(K)_{\mathrm{tors}}[3^\infty]\) is trivial, then by Proposition~\ref{prop:Z3structurenonQ-3}, we have the claim in the theorem.

    Suppose that \(P \in E(K_{cyc})\) has order \(27\), by Lemma~\ref{lemma:p-twist}, we have \(E\) or \(E^{(K)}\) possesses a point of order \(27\). By Avc\i's theorem, i.e., Theorem~\ref{theorem:Avci}, and the remark in \cite[pp.~168]{AVCI2026153}, only the elliptic curve labeled by \(\mathrm{27a4}\) has a \(27\)-torsion point in the cyclotomic \(\mathbb{Z}_3\)-extension of \(\mathbb{Q}\). By \cite{chou2021}, for the curve labeled by \(\mathrm{27a4}\), \(E(\mathbb{Q})_\mathrm{tors}[3^\infty]=\mathbb{Z}/3\mathbb{Z}\). Hence either  \(E\) is isomorphic to \(\mathrm{27a4}\) or \(E^{(K)}\) is isomorphic to \(\mathrm{27a4}\). It is easy to show that, \(E(K)_\mathrm{tors}[3^\infty]=\mathbb{Z}/3\mathbb{Z}\). In fact, since \(3\) is odd, then \(E(K)[3]\cong E(\mathbb{Q})[3]\oplus E^{(K)}(\mathbb{Q})[3]\). Since \(K\neq\mathbb{Q}(\sqrt{-3})\), \(E(K)[3]\) is rank at most \(1\). Since the curve \(\mathrm{27a4}\) has \(3\)-torsion point over \(\mathbb{Q}\) and then the claim is done.
    Therefore, by Proposition~\ref{prop:Z3structurenonQ-3}, this case is done.

    For all other elliptic curves not isomorphic to \(\mathrm{27a4}\), if \(E(K)\) has a \(9\)-torsion group, then it will not have a \(27\)-torsion group. This case is done. If \(E(K)\) only has a \(3\)-torsion group, then either the growth happens or the growth does not happen. If the growth happens, this case is also done by Proposition~\ref{prop:Z3structurenonQ-3}. If it does not happen, then this is also done by the same proposition.
\end{proof}

Now suppose that \(E(K)_\mathrm{tors}[3^\infty]\cong \mathbb{Z}/3\mathbb{Z}\) and suppose \(P\) is a generator. By the previous theorem, \(E(K^1_{cyc})_{\mathrm{tors}}[3^\infty]\cong \mathbb{Z}/3\mathbb{Z}\) or \(\mathbb{Z}/9\mathbb{Z}\). We now use Kummer theory and Galois cohomology to give a criterion.

Consider 
\[
0 \rightarrow E[3] \rightarrow E(\overline{\mathbb{Q}}) \xrightarrow{\times 3} E(\overline{\mathbb{Q}}) \rightarrow 0
\]
and taking \(G_{K^1_{cyc}}=\Gal{\overline{\mathbb{Q}}}{K^1_{cyc}}\)-invariants gives the following long exact sequence
\[
0 \rightarrow E(K^1_{cyc})[3] \rightarrow E({K^1_{cyc}}) \xrightarrow{\times 3} E({K^1_{cyc}}) \xrightarrow{\delta_{K^1_{cyc}}}  H^1(K^1_{cyc}, E[3]).
\]
Now suppose that \(E(K^1_{cyc})_{\mathrm{tors}}[3^\infty]\cong \mathbb{Z}/9\mathbb{Z}\), then \(P=3Q\) for some \(Q\) of order \(9\). It follows that \(P\) lies in \(\im(\times 3)\) and hence lies in \(\ker\delta_{K^1_{cyc}}\), which means \(\delta_{K^1_{cyc}}(P)=0\). Conversely, if \(\delta_{K^1_{cyc}}(P)=0\), then \(P=3Q\) for some \(Q\in E(K^1_{cyc})\). Since \(P\) is non-trivial, \(Q\) is a point of order \(9\). Hence we have

\begin{proposition}\label{prop:3-kummertheory}
    Let \(E/\mathbb{Q}\) be an elliptic curve and \(K\) be a quadratic field which is not \(\mathbb{Q}(\sqrt{-3})\). Suppose that \(E(K)_{\mathrm{tors}}[3^\infty]\cong \mathbb{Z}/3\mathbb{Z}\) with a generator \(P\). Then 
    \[
    E(K^1_{cyc})[3^\infty]\cong \mathbb{Z}/9\mathbb{Z} \iff \delta_{K_{cyc}^1}(P)=0.
    \]
    In particular, if we further assume that \(E\) and  \(E^{(K)}\) is not isomorphic to \(\mathrm{27a4}\), then
    \[
    E(K_\infty )[3^\infty]\cong \mathbb{Z}/9\mathbb{Z} \iff \delta_{K_{cyc}^1}(P)=0.
    \]
    Here, \(K_\infty\) is the compositum of all \(\mathbb{Z}_3\)-extensions of \(K\).
\end{proposition}

\subsubsection*{The prime \(q=2\)}
\begin{proposition}\label{prop:2-torsion-KneqQsqrt-3}
Let \(K\) be a quadratic field, and let \(K_\infty\) be the compositum of all \(\mathbb{Z}_3\)-extensions of \(K\). For an elliptic curve \(E/\mathbb{Q}\), the \(2\)-primary torsion subgroup satisfies the following. For every single \(\mathbb{Z}_3\)-extension \(M\), \(E(M)_{\mathrm{tors}}[2^\infty]=E(M_1)_\mathrm{tors}[2^\infty]\), where \(M_1\) is the first layer of \(M\). Hence \(E(K_\infty)_\mathrm{tors}[2^\infty]=E(K_\infty^{(1)})_\mathrm{tors}[2^\infty]\).

Moreover, we have the following.
\begin{enumerate}
  \item If \(E(K)[2] \neq 0\), then \(E(K_\infty)_{\mathrm{tors}}[2^\infty] = E(K)_{\mathrm{tors}}[2^\infty]\).
  \item If \(E(K)[2] = 0\), then either
  \[
  E(K_\infty)_{\mathrm{tors}}[2^\infty] = 0,
  \]
  or there exists an integer \(a \ge 1\) such that
  \[
  E(K_\infty)_{\mathrm{tors}}[2^\infty] \cong (\mathbb{Z}/2^a\mathbb{Z})^2.
  \]
  Moreover, the latter case occurs if and only if the compositum \(K L\) is a cubic extension of \(K\) contained in \(K_\infty\), where \(L = \mathbb{Q}(E[2])\) is the field generated by the \(2\)-torsion points of \(E\).
\end{enumerate}
\end{proposition}

\begin{proof}
Since \(q=2\) and \(p=3\) are distinct, Lemma~\ref{lemma:growthpattern=ellnotp} applies to every \(\mathbb{Z}_3\)-extension. For \(d_j=\ord_{3^j}(2)\), we have \(d_1=2\) and \(d_j\ge6\) for \(j\ge2\). Hence growth can occur only from the base field to the first layer, and every non-zero new block requires two generators. This proves the assertion for a single \(\mathbb{Z}_3\)-extension.

We now pass to the compositum. If \(K\) is real, then \(K_\infty/K\) is itself a \(\mathbb{Z}_3\)-extension and there is nothing more to prove. Suppose that \(K\) is imaginary, so that \(K_\infty=K_{cyc}K_{anti}\). Let \(P\in E(K_\infty)_{\mathrm{tors}}[2^\infty]\). Then
\[
P\in E(K_{cyc}^nK_{anti}^m)
\]
for some \(n,m\ge0\). Apply the single-extension assertion first to \(K_{cyc}K_{anti}^m/K_{anti}^m\)
and then to \(K_{cyc}^1K_{anti}/K_{cyc}^1.\)
It follows that \(P\in E(K_{cyc}^1K_{anti}^1)\). Since \(K_{cyc}^1K_{anti}^1\subseteq K_\infty^{(1)}\), we obtain
\[
E(K_\infty)_{\mathrm{tors}}[2^\infty]=E(K_\infty^{(1)})_{\mathrm{tors}}[2^\infty].
\]

If \(E(K)[2]\neq0\), a non-zero new block would require two additional generators, whereas a subgroup of \(E[2^\infty]\) requires at most two generators. Thus no growth can occur in either direction, and
\[
E(K_\infty)_{\mathrm{tors}}[2^\infty]=E(K)_{\mathrm{tors}}[2^\infty].
\]

Assume now that \(E(K)[2]=0\), and let \(L=\mathbb{Q}(E[2])\). If the \(2\)-primary torsion grows in \(K_\infty\), then the image of \(\Gal{K_\infty}{K}\) on \(E[2]\) is a non-trivial \(3\)-subgroup of \(\GL_2(\mathbb{F}_2)\cong S_3\). It is therefore cyclic of order \(3\). Consequently,
\[
[KL:K]=3,
\]
and \(KL/K\) is the first layer of a \(\mathbb{Z}_3\)-extension contained in \(K_\infty\). Lemma~\ref{lemma:growthpattern=ellnotp} then gives
\[
E(KL)_{\mathrm{tors}}[2^\infty]\cong(\mathbb{Z}/2^a\mathbb{Z})^2
\]
for some \(a\ge1\). Once the full \(2\)-torsion is present, no further block can be added in either \(\mathbb{Z}_3\)-direction, again because the resulting group would require more than two generators. Hence the same group is \(E(K_\infty)_{\mathrm{tors}}[2^\infty]\).

Conversely, if \([KL:K]=3\) and \(KL\subseteq K_\infty\), then \(KL/K\) is the first layer of some \(\mathbb{Z}_3\)-extension of \(K\). Since \(E[2]\subseteq E(KL)\), Lemma~\ref{lemma:growthpattern=ellnotp} gives \(E(KL)_{\mathrm{tors}}[2^\infty]\cong(\mathbb{Z}/2^a\mathbb{Z})^2\) for some \(a\ge1\), and the preceding argument shows that this group does not grow further in \(K_\infty\). This proves the equivalence and the proposition.
\end{proof}

\begin{remark}
    If \(K\neq \mathbb{Q}(i)\), then by Lemma~\ref{lemma:rootsofunity-p=3}, we have \(\mu_{2^\infty}(K_\infty)=\mu_2\), and by Lemma~\ref{lemma:rootsofunity}, we have \(a=1\).
\end{remark}

\subsection{The case \(K = \mathbb{Q}(\sqrt{-3})\)}

By Lemma~\ref{lemma:rootsofunity-p=3}, if \(F/K\) is contained in \(K_{cyc}\), then the group \(\mu(F)\) can be arbitrarily large. Nevertheless, we now examine each of the prime cases handled in the previous subsection to determine whether the arguments carry over to the present setting.

First of all, Proposition~\ref{prop:reducetocyc} fails only when \(q=3\). Hence
\[
E(K_\infty)_{\mathrm{tors}}[q^\infty]=E(K_{cyc})_{\mathrm{tors}}[q^\infty]
\] 
holds for every odd prime \(q\ge 5\).

\begin{itemize}
    \item For \(q = 5, 11, 17\), the argument relied only on the fact that
    \[
    \gcd([F:\mathbb{Q}], \varphi(q^k)) = 2
    \]
    for all \(k \ge 1\), which holds for every intermediate subfield \(F/K\) of \(K_\infty/K\). Lemma~\ref{lemma:avci-lemma3.1} therefore applies without any restriction on \(K\). Hence the same conclusion holds for \(K = \mathbb{Q}(\sqrt{-3})\):
    \[
    E(K_\infty)_{\mathrm{tors}}[q^\infty] = E(K)_{\mathrm{tors}}[q^\infty].
    \]

    \item For \(q = 13, 19, 43, 67, 163\), the argument uses a combination of Proposition~\ref{prop:growth-pattern-q>3}, Lemma~\ref{lemma:ell-twist}, and Theorem~\ref{theorem:chouQinfty}(iii). None of these results imposes any restriction on \(K\). Hence the conclusion also holds for \(K = \mathbb{Q}(\sqrt{-3})\):
    \[
    E(K_\infty)_{\mathrm{tors}}[q^\infty] = E(K)_{\mathrm{tors}}[q^\infty].
    \]

    \item For \(q = 7\), Proposition~\ref{prop:p=3,q=7}  does not exclude \(K=\mathbb{Q}(\sqrt{-3})\). Thus
    \[
    E(K_\infty)_{\mathrm{tors}}[7^\infty] = E(K_{cyc}^1)_{\mathrm{tors}}[7^\infty].
    \]

    \item For \(q = 2\), Proposition~\ref{prop:2-torsion-KneqQsqrt-3} holds for all quadratic fields \(K\). Its proof makes no assumption on \(K\). Thus the result remains true for \(K = \mathbb{Q}(\sqrt{-3})\).
\end{itemize}

\begin{theorem}\label{theorem:p=3=3}
    Theorem \ref{theorem:p=3not-3} holds for \(K=\mathbb{Q}(\sqrt{-3})\) except for \(q=3\).
\end{theorem}

\subsubsection*{The prime \(q=3\)}

\begin{lemma}\label{lemma:firstlayerp=3}
    Let \(E/\mathbb{Q}\) be an elliptic curve and \(K\) be a number field. For any \(\mathbb{Z}_3\)-extension \(L/K\), we have
    \[
    E(L)_{\mathrm{tors}}[3^n]=E(L_n)_\mathrm{tors}[3^n],
    \]
    where \(L_n\) is the \(n\)-th layer of \(L\) for \(n\ge1 \).
\end{lemma}

\begin{proof}
    Let \(P\in E(L)[3^n]\), and let \(\gamma\) be a topological generator of \(\Gal{L}{K}\). Choose an extension \(\widetilde\gamma\in G_K\) and let \(A\) be its image in \(\GL_2(\mathbb{Z}/3^n\mathbb{Z})\). The degree \([K(P):K]\) is the length of the orbit of \(P\) under \(\gamma\). It is a power of \(3\) dividing the order of \(A\).

    The \(3\)-part of the order of the reduction of \(A\) in \(\GL_2(\mathbb{F}_3)\) divides \(3\). Moreover, the kernel of
    \[
    \GL_2(\mathbb{Z}/3^n\mathbb{Z})\longrightarrow \GL_2(\mathbb{F}_3)
    \]
    has exponent dividing \(3^{n-1}\): this follows by iterating the binomial congruence \((I+3B)^{3^r}\equiv I\pmod{3^{r+1}}\). Hence the \(3\)-part of the order of \(A\) divides \(3^n\). Therefore \([K(P):K]\mid 3^n\), and the uniqueness of the layers of \(L/K\) gives \(K(P)\subseteq L_n\). We are done.
\end{proof}

Since \(K_{cyc}\) contains \(\mu_{3^\infty}\), \(E(K_\infty)_\mathrm{tors}[3^\infty]\) is not a cyclic group in general. The following lemma shows that there are restrictions on the structure of \(E(F)_\mathrm{tors}[3^\infty]\) for a Galois subextension \(K_\infty/F/\mathbb{Q}\).

\begin{lemma}\label{lemma:p=3,a-ble3}
    Let \(E/\mathbb{Q}\) be an elliptic curve and \(K=\mathbb{Q}(\sqrt{-3})\). Let \(F\) be any finite subextension of \(K_\infty/K\) that is Galois over \(\mathbb{Q}\), where \(K_\infty\) is the compositum of all \(\mathbb{Z}_3\)-extensions of \(K\). Suppose that \[E(F)_{\mathrm{tors}}[3^\infty]\cong \mathbb{Z}/3^a\mathbb{Z}
    \times \mathbb{Z}/3^b\mathbb{Z}\] with \(a\ge b \ge 0\). Then \(a-b\le 3\).
\end{lemma}

\begin{proof}
    Suppose otherwise, \(a-b>3\). Then by Lemma~\ref{lemma:admits-n-isogeny}, \(E\) admits an \(81\)-isogeny which is impossible by Theorem~\ref{theorem:isogeniesoverQ}.
\end{proof}

A quick application is the following proposition.

\begin{proposition}\label{prop:p=3-infinite-dichotomy}
    Let \(E/\mathbb{Q}\) be an elliptic curve and \(K=\mathbb{Q}(\sqrt{-3})\), with \(K_\infty\) the compositum of all \(\mathbb{Z}_3\)-extensions. Either \(E(K_\infty)_\mathrm{tors}[3^\infty]\) is finite or \(E[3^\infty]\subseteq E(K_\infty)\).
\end{proposition}

\begin{proof}
     Suppose that \(E(K_\infty)_{\mathrm{tors}}[3^\infty]\) is infinite. Then it contains a subgroup isomorphic to \(\mathbb{Q}_3/\mathbb{Z}_3\). Take a point \(P_n\) of order \(3^n\), with \(n\) arbitrarily large, and let \(M_n\) be the Galois closure of \(K(P_n)\) over \(\mathbb{Q}\). Since \(K_\infty/\mathbb{Q}\) is Galois, we have \(M_n\subseteq K_\infty\). Lemma~\ref{lemma:p=3,a-ble3} shows that \(E(M_n)\) contains a subgroup isomorphic to \(\mathbb{Z}/3^n\mathbb{Z}\times\mathbb{Z}/3^{n-3}\mathbb{Z}.\)
     Thus, for every \(m\ge1\), taking \(n\ge m+3\) gives \(E[3^m]\subseteq E(M_n)\subseteq E(K_\infty)\). Hence \(E[3^\infty]\subseteq E(K_\infty)\).
\end{proof}

\begin{theorem}\label{theorem:p=3singleZpfinite}
     Let \(E/\mathbb{Q}\) be an elliptic curve and \(K=\mathbb{Q}(\sqrt{-3})\), with \(K_\infty\) the compositum of all \(\mathbb{Z}_3\)-extensions. Let \(L/K\) be a \(\mathbb{Z}_3\)-extension. Then \(E(L)_\mathrm{tors}[3^\infty]\) is finite.
\end{theorem}

\begin{proof}
    By the previous proposition, we only consider the case where \(E[3^\infty] \subseteq E(K_\infty)\). Consider the Galois representation associated with \(E[3^\infty]\):
    \[
    \rho_{3^\infty}: \Gal{\overline{\mathbb{Q}}}{K} \rightarrow \GL(E[3^\infty]).
    \]
    Since \(E[3^\infty] \subseteq E(K_\infty)\), \(\Gal{\overline{\mathbb{Q}}}{K_\infty} \subseteq \ker \rho\) and hence induces
    \[
    \rho: \Gal{K_\infty}{K} \rightarrow \GL(E[3^\infty])\cong \GL_2(\mathbb{Z}_3).
    \]
    This means \(\im \rho\) is a quotient of \(\mathbb{Z}_3^2\), hence a pro-\(3\) abelian group. By Serre's open image theorem \cite[Th\'eor\`eme~3, p.~299]{Serre1972}, if \(E\) is non-CM, then \(\rho(\Gal{\overline{\mathbb{Q}}}{\mathbb{Q}})\) is open in \(\GL_2(\mathbb{Z}_3)\), so the image is not abelian. Hence \(\rho(\Gal{\overline{\mathbb{Q}}}{K})\) is also not abelian. Therefore, under our assumption, \(E\) has CM.

    We claim that \(E\) has CM exactly by \(K\). Suppose not, so that \(E\) has CM by \(F\neq K\). Take \(\sigma \in \Gal{\overline{\mathbb{Q}}}{K}\) such that \(\tau:=\sigma|_F \in \Gal{F}{\mathbb{Q}}\) is non-trivial. Also choose \(g \in \Gal{\overline{\mathbb{Q}}}{KF}\) such that \(\rho(g)\) is not a scalar matrix and \(\rho(g)\neq \rho(g)^\tau\) since the image is open. On the one hand, \(\rho(\sigma g \sigma^{-1})=\rho(g)^\tau\) in general. On the other hand, since \(\rho(\Gal{\overline{\mathbb{Q}}}{K})\) is abelian, \(\rho(\sigma g \sigma^{-1})=\rho(g)\). This contradicts \(\rho(g)\neq \rho(g)^\tau\). Hence \(F=K\).

    By CM theory, the representation can be written as 
    \[
    \rho :\Gal{K_\infty}{K} \rightarrow (\mathcal{O}\otimes \mathbb{Z}_3)^\times \cong \mathbb{Z}_3^2\times (finite)
    \]
    where \(\mathcal{O}\) is the CM order.
    Since the image is open, the \(\mathbb{Z}_3\)-rank of \(\im \rho\) is two and hence \(\ker \rho\) has \(\mathbb{Z}_3\)-rank \(0\). Since \(\Gal{K_\infty}{K}\) is torsion-free, \(\rho\) is injective.

    Now suppose \(\Gal{K_\infty}{L}=\overline{\langle \gamma\rangle}\), then \(E(L)_\mathrm{tors}[3^\infty]=E[3^\infty]^{\Gal{K_\infty}{L}}=\ker (\rho(\gamma)-1: E[3^\infty]\rightarrow E[3^\infty]) \cong \coker (\rho(\gamma)-1: T_3(E)\rightarrow T_3(E))\). Since \(\rho\) is injective and \(3\) is ramified in \(K\), \(\rho(\gamma)-1\neq 0\) in \(\mathcal{O}\otimes \mathbb{Z}_3\), which means the cokernel is finite. Hence \(E(L)_{\mathrm{tors}}[3^\infty]\) is finite.
\end{proof}

Theorem~\ref{theorem:p=3singleZpfinite} and its proof showed that it is necessary to distinguish two cases: the case where \(E\) has CM and the case where \(E\) does not have CM. We deal with the case where \(E\) has CM by an order in \(K\).

\begin{theorem}\label{theorem:CMclass,p=3}
    Let \(E/\mathbb{Q}\) be an elliptic curve and \(K=\mathbb{Q}(\sqrt{-3})\), with \(K_\infty\) the compositum of all \(\mathbb{Z}_3\)-extensions. Then \(E[3^\infty]\subseteq E(K_\infty)\) if and only if \(E\) satisfies all of the following conditions:
    \begin{itemize}
        \item \(E\) has CM by an order in \(K\);

        \item \(E\) has good reduction outside \(3\).
    \end{itemize}

    Moreover, all elliptic curves satisfying these conditions are isomorphic to one of the following curves with Cremona label: \(\mathrm{27a1}\), \(\mathrm{27a2}\), \(\mathrm{27a3}\), \(\mathrm{27a4}\), \(\mathrm{243a1}\), \(\mathrm{243a2}\), \(\mathrm{243b1}\), \(\mathrm{243b2}\), which are exactly the curves listed in \cite[Table~1]{xu2026torsiongrowthrationalelliptic}.
\end{theorem}

\begin{proof}
    Suppose \(E[3^\infty]\subseteq E(K_\infty)\), by the proof of Theorem~\ref{theorem:p=3singleZpfinite}, we know \(E\) has CM by an order in \(K\). Since \(K(E[3^\infty])/K\) is contained in \(K_\infty\), which is unramified outside \(3\), by the N\'{e}ron--Ogg--Shafarevich theorem, \(E\) has good reduction outside \(3\) over \(K\). Now it remains to show \(E\) has good reduction outside \(3\) over \(\mathbb{Q}\). Now \(K/\mathbb{Q}\) is ramified only at \(3\), suppose that \(q\neq 3\), then \(q\) splits or is inert in \(K\). If \(q\) splits, then local field does not extend at \(q\). Then good reduction of \(E/K\) at primes over \(q\) implies that \(E/\mathbb{Q}\) has good reduction at \(q\). If \(q\) is inert, then \(\mathfrak{q}\) is the unique prime lying over \(q\). The inertia group of \(K_{\mathfrak{q}}\) equals that of \(\mathbb{Q}_q\). Since good reduction holds if and only if the action of the inertia groups is trivial on the Tate module, \(E/\mathbb{Q}\) has good reduction at \(q\). Therefore \(E/\mathbb{Q}\) has good reduction outside \(3\).
    
    Conversely, suppose that \(E\) has CM by an order in \(K=\mathbb{Q}(\sqrt{-3})\) and has good reduction outside \(3\) over \(\mathbb{Q}\). Then \(E/K\) has good reduction outside the prime above \(3\). Corollary~4.2 of \cite{xu2026torsiongrowthrationalelliptic} therefore gives \(E[3^\infty]\subseteq E(K_\infty)\).

    In \cite{xu2026torsiongrowthrationalelliptic}, they showed that these eight curves labeled by \(\mathrm{27a1}\), \(\mathrm{27a2}\), \(\mathrm{27a3}\), \(\mathrm{27a4}\), \(\mathrm{243a1}\), \(\mathrm{243a2}\), \(\mathrm{243b1}\), \(\mathrm{243b2}\) satisfy the condition that \(E\) has CM by \(K\) and that their discriminants are supported at \(3\). We now show that there are no other curves.

    Now suppose that \(E\) has CM by an order \(\mathcal{O}\) in \(K\), then \(\mathcal{O}=\mathcal{O}_f=\mathbb{Z}+f\mathcal{O}_K\) for \(f=1,2,3\). For \(f \ge 4\), the class number of \(\mathcal{O}_f\)  is not \(1\) which are not our cases.

    \textbf{Case: \(f=1\)}. In this case, the discriminant of \(\mathcal{O}_f\) is \(-3\) and the \(j\)-invariant is \(0\). The Cremona--Lingham algorithm~\cite{CremonaLingham2007}, as implemented in SageMath by \texttt{egros\_from\_j(0,[3])}, yields exactly the following six \(\mathbb Q\)-isomorphism classes:
    \[ 
    \mathrm{27a1},\ \mathrm{27a3},\ \mathrm{243a1},\ \mathrm{243a2},\ \mathrm{243b1},\ \mathrm{243b2}. 
    \]
    See also the SageMath documentation for \texttt{egros\_from\_j}.

    \textbf{Case: \(f=2\)}. In this case, the discriminant of \(\mathcal{O}_f\) is \(-12\) and the \(j\)-invariant is \(54000\). By Cremona--Pacetti~\cite[Introduction]{CremonaPacetti2019}, if a prime \(p\) divides the discriminant of the CM order, then every elliptic curve with that endomorphism ring has additive reduction at \(p\). Since \(2\mid -12\), no such curves exist in this case.

    \textbf{Case: \(f=3\)}. In this case, the discriminant of \(\mathcal{O}_f\) is \(-27\) and the \(j\)-invariant is \(-12288000\). Again by the Cremona--Lingham algorithm~\cite{CremonaLingham2007} and Sage, this computation yields exactly two \(\mathbb Q\)-isomorphism classes:
    \[
    \mathrm{27a2},\ \mathrm{27a4}.
    \]
    Indeed, the SageMath computation \texttt{egros\_from\_j(-12288000,[3])} returns precisely these two curves. Thus the case \(f=3\) contributes exactly two \(\mathbb Q\)-isomorphism classes.

    Hence all eight curves are found.
\end{proof}

In \cite{xu2026torsiongrowthrationalelliptic}, they also compute the group structure \(E(L)_{\mathrm{tors}}[3^\infty]\) for all \(\mathbb{Z}_3\)-extensions \(L/K\), see \cite[Theorem~4.7]{xu2026torsiongrowthrationalelliptic}.

We now give the growth pattern. The following lemma is a non-cyclic version of Lemma~\ref{lemma:grothpatternp=ell,depth} at \(p=3\).

\begin{lemma}\label{lemma:grothpatternp=ell=3,depth,non-cyclic}
    Let \(E/\mathbb{Q}\) be an elliptic curve, let \(K=\mathbb{Q}(\sqrt{-3})\), and let \(L/K\) be a \(\mathbb{Z}_3\)-extension with intermediate subfields \(L_n\) for \(n\ge 1\), and set \(L_0=K\). Suppose that \(E(L)_{\mathrm{tors}}[3]=\mathbb{Z}/3\mathbb{Z}\times \mathbb{Z}/3\mathbb{Z}\). Then for \(n\ge 1\),
    \[
     G_n:=\frac{E(L_{n+1})_{\mathrm{tors}}[3^\infty]}{ E(L_{n})_{\mathrm{tors}}[3^\infty]}\in \{ 1 ,\ \mathbb{Z}/3\mathbb{Z},\ \mathbb{Z}/3\mathbb{Z} \times \mathbb{Z}/3\mathbb{Z}\}.
    \]
    Moreover, if \(E[3]\subseteq E(K)\), this holds also for \(n=0\).
    
    Furthermore, there exists a nonnegative integer \(s\ge 0\) such that when \(n\ge s\), \(G_n=1\), and when \(n<s\), \(G_n\neq 1\). In other words, the growth only appears in the first \(s\) layers.
\end{lemma}

\begin{proof}
    As in Lemma~\ref{lemma:grothpatternp=ell,depth}, the proof is algebraic. By Theorem~\ref{theorem:p=3singleZpfinite}, \(E(L)_{\mathrm{tors}}[3^\infty]\) is finite. By our assumption on \(E(L)_\mathrm{tors}[3]\), we have \(E(L)_\mathrm{tors}[3^\infty]\cong \mathbb{Z}/3^a\mathbb{Z} \times \mathbb{Z}/3^b\mathbb{Z}\) for \(a\ge b >0\).
    
    Let \(\gamma\) be a topological generator of \(\Gamma:=\Gal{L}{K}\) and \(\sigma_n:=\gamma^{3^n}\). The action of \(\gamma\) on \(E(L)_\mathrm{tors}[3^\infty]\) has finite \(3\)-power order, say \(3^r\). Set \(\Gamma_n=\Gamma^{3^n}\) for \(n\ge 0\) and \(M=E(L)_{\mathrm{tors}}[3^\infty]\), with \(M_n=M^{\Gamma_n}=E(L_n)_\mathrm{tors}[3^\infty]\).

    Now \(\Gal{L_{n+1}}{L_n}=\langle \overline{\sigma_n}\rangle\) has a natural action on \(M_{n+1}\) via \(\overline{\sigma_n}:M_{n+1}\rightarrow M_{n+1}\), and \(M_n=\ker(\overline{\sigma_n}-1)\). Then
    \[
    G_n= M_{n+1}/M_n\cong \im (\overline{\sigma_n}-1).
    \]
    We have \(E[3] \subseteq M_n[3]\) for \(n\ge 1\) by Lemma~\ref{lemma:firstlayerp=3}. This means \(\overline{\sigma_n}|_{M_{n+1}[3]}\) is trivial for \(n\ge1\). Hence by Lemma~\ref{lemma:analgberaiclemmap=3} below, \((\overline{\sigma_n}-1)M_{n+1} \subseteq M_{n+1}[3]\). Hence \(3G_n\) is trivial.

    Therefore, \(G_n\) is isomorphic to one of the following: \(1\), \(\mathbb{Z}/3\mathbb{Z}\), or \(\mathbb{Z}/3\mathbb{Z} \times \mathbb{Z}/3\mathbb{Z}\) for \(n\ge1\).

    Now if \(E[3]\subseteq E(K)\), by Lemma~\ref{lemma:analgberaiclemmap=3}, we have \(3G_0=1\).

    Let \(s=r\). If \(n\ge s\), then \(\gamma^{3^n}\) acts trivially on \(M\), so \(M_n=M\) and \(G_n=1\). If \(n<s\), apply Lemma~\ref{lemma:fixed-point-growth} to the non-trivial automorphism \(g=\gamma^{3^n}\). It gives
    \[
    M_n=M^g\subsetneq M^{g^3}=M_{n+1},
    \]
    and hence \(G_n\neq1\). This proves the last assertion.
\end{proof}

\begin{lemma}\label{lemma:analgberaiclemmap=3}
    Let \(M\) be a finite abelian \(3\)-primary group, i.e. \(M=M[3^\infty]\). Let \(\sigma \in \Aut(M)\) with \(\sigma^3=1\) and \( \sigma|_{M[3]}=1\). Then we have \((\sigma-1)M \subseteq M[3]\).
\end{lemma}

\begin{proof}
    Let \(T:=\sigma-1\), then \((1+T)^3=1\) by \(\sigma^3=1\). Hence \(T^3+3T^2+3T=0\) and \(T^3=-3T(T+1)=-3\sigma T\). Now \((\sigma-1)M\subseteq M[3]\) is equivalent to \(3Tx=0\) for \(x\in M\).

    Let \(x \in M\) and \(y=Tx\). Suppose that \(r\) is the smallest integer such that \(3^r y=0\). If \(r=0\), there is nothing to discuss. Now let \(r \ge 1\), and \(3^{r-1}y \in M[3]\). Since \(\sigma|_{M[3]}=1\), we have \(\sigma(3^{r-1}y)=3^{r-1}y\) and hence \(T(3^{r-1}y)=3^{r-1}Ty=0\). Hence \(Ty\) has order \(3^{r-1-t}\) for some \(t \ge 0\). Suppose that \(r\ge 2\), and similarly \(T^2y\) has order \(3^{r-2-t}\) for some \(t\ge 0\). However, \(T^2y=T^3x=-3\sigma Tx=-3\sigma y\). Since \(\sigma y\) has the same order as \(y\), \(T^2y=-3\sigma y\) has exact order \(3^{r-1}\), which contradicts the order of \(T^2y\). This means \(r\le1\), and hence \(y=Tx\in M[3]\). This proves the lemma.
\end{proof}

\begin{theorem}\label{theorem:growth-patternQ(-3),-3}
    Let \(E/\mathbb{Q}\) be an elliptic curve and \(K=\mathbb{Q}(\sqrt{-3})\). Let \(L/K\) be a \(\mathbb{Z}_3\)-extension with intermediate fields \(L_n\). Set \(L_0=K\). Then the growth pattern is one of the following:

    \begin{itemize}
    \item \(E(L)_\mathrm{tors}[3]\) is trivial if and only if \(E(K)_\mathrm{tors}[3]\) is trivial.
    \item If \(E(L)_\mathrm{tors}[3]\cong \mathbb{Z}/3\mathbb{Z}\), then the case is exactly the same as Theorem~\ref{theorem:growthppterQneqQ-3} when \(L=K_{cyc}\). In this case,  \(E^{(K)}\)  is  isomorphic to \(\mathrm{27a2}\) if \(E\) is isomorphic to \(\mathrm{27a4}\). 
    
    For any \(L\), suppose that \(E(L)_{\mathrm{tors}}[3^\infty]\cong\mathbb{Z}/3^s \mathbb{Z}\) and \(E(K)_\mathrm{tors}[3^\infty]\cong \mathbb{Z}/3^a \mathbb{Z}\) for \(s\ge a \ge 1\), then \(E(L_n)_\mathrm{tors}[3^\infty]\cong \mathbb{Z}/3^{a+n}\mathbb{Z}\) for \(0 \le n\le s-a\) and \(E(L_n)_\mathrm{tors}[3^\infty]=E(L)_\mathrm{tors}[3^\infty]\) for \(n\ge s-a\).

    \item If \(E(L)_\mathrm{tors}[3]\cong \mathbb{Z}/3\mathbb{Z} \times \mathbb{Z}/3\mathbb{Z}\):

    \begin{itemize}
        \item If \(E(K)_{\mathrm{tors}}[3]\cong \mathbb{Z}/3\mathbb{Z}\times \mathbb{Z}/3 \mathbb{Z}\), then there exists \(s\ge 0\) such that for \(n<s\), the quotient group \(E(L_{n+1})_\mathrm{tors}[3^\infty]/E(L_n)_\mathrm{tors}[3^\infty]\) is isomorphic to \(\mathbb{Z}/3\mathbb{Z}\) or \(\mathbb{Z}/3\mathbb{Z}\times \mathbb{Z}/3\mathbb{Z}\), and for \(n\ge s\), \(E(L_{n+1})_\mathrm{tors}[3^\infty]=E(L_n)_\mathrm{tors}[3^\infty]=E(L)_{\mathrm{tors}}[3^\infty]\).

        \item If \(E(K)[3]\cong \mathbb{Z}/3\mathbb{Z}\):
            \begin{itemize}
                \item If \(E(K)[3^\infty]\cong \mathbb{Z}/3\mathbb{Z}\), then \(E(L_1)[3^\infty]\cong \mathbb{Z}/3\mathbb{Z} \times \mathbb{Z}/3\mathbb{Z}\) or \(E(L_1)[3^\infty]\cong \mathbb{Z}/9\mathbb{Z} \times \mathbb{Z}/3\mathbb{Z}\). There exists \(s\ge 1\) such that for \(1\le n<s\), the quotient group \(E(L_{n+1})_\mathrm{tors}[3^\infty]/E(L_n)_\mathrm{tors}[3^\infty]\) is isomorphic to \(\mathbb{Z}/3\mathbb{Z}\) or \(\mathbb{Z}/3\mathbb{Z}\times \mathbb{Z}/3\mathbb{Z}\), and for \(n\ge s\), \(E(L_{n+1})_\mathrm{tors}[3^\infty]=E(L_n)_\mathrm{tors}[3^\infty]=E(L)_{\mathrm{tors}}[3^\infty]\).

                \item If \(E(K)[3^\infty]\cong \mathbb{Z}/9\mathbb{Z}\), then \(E(L_1)[3^\infty]\cong \mathbb{Z}/9 \mathbb{Z} \times \mathbb{Z}/3 \mathbb{Z}\). There exists \(s\ge 1\) such that for \(1\le n<s\), the quotient group \(E(L_{n+1})_\mathrm{tors}[3^\infty]/E(L_n)_\mathrm{tors}[3^\infty]\) is isomorphic to \(\mathbb{Z}/3\mathbb{Z}\) or \(\mathbb{Z}/3\mathbb{Z}\times \mathbb{Z}/3\mathbb{Z}\), and for \(n\ge s\), \(E(L_{n+1})_\mathrm{tors}[3^\infty]=E(L_n)_\mathrm{tors}[3^\infty]=E(L)_{\mathrm{tors}}[3^\infty]\).
            \end{itemize}
    Moreover, if \(L\cap K_{cyc}=K_{cyc}^n\) for \(n\ge 1\), the case \(E(K)[3]\cong \mathbb{Z}/3\mathbb{Z}\) and \(E(L)[3]\cong \mathbb{Z}/3\mathbb{Z} \times \mathbb{Z}/3\mathbb{Z}\) is impossible.
    \end{itemize}

    \end{itemize}
\end{theorem}

\begin{proof}
    If \(E(K)_\mathrm{tors}[3]\) has rank \(0\), then \(E(L)_\mathrm{tors}[3]\) is trivial by Lemma~\ref{lemma:grothpatternp=ell}.
    
    If \(E(L)_\mathrm{tors}[3]\) has rank \(1\), then it is finite and cyclic by Theorem~\ref{theorem:p=3singleZpfinite}. Hence Lemma~\ref{lemma:grothpatternp=ell,depth} can be applied. In particular, when \(L=K_{cyc}\), the growth pattern is exactly the same as that in Theorem~\ref{theorem:growthppterQneqQ-3}. In fact, when \(E(K_{cyc})\) contains a point of order \(27\), \(E\) admits a rational \(27\)-isogeny. Then by \(X_0(27)\), we know that \(j(E)=-12288000\), hence such elliptic curves either are isomorphic to \(\mathrm{27a4}\) or are quadratic twists of \(\mathrm{27a4}\). When \(K = \mathbb{Q}(\sqrt{-3})\), if \(E\ncong \mathrm{27a4}\) has a point of order \(27\), then \(E\) must be the quadratic twist of \(\mathrm{27a4}\) by \(-3\), which is exactly \(\mathrm{27a2}\).

        Now we deal with the rank \(2\) case. \textbf{Case 1: \(E(K)_{\mathrm{tors}}[3]\cong \mathbb{Z}/3\mathbb{Z}\times \mathbb{Z}/3 \mathbb{Z}\)}. It follows directly from Lemma~\ref{lemma:grothpatternp=ell=3,depth,non-cyclic}. \textbf{Case 2: \(E(K)_{\mathrm{tors}}[3]\cong \mathbb{Z}/3\mathbb{Z}\)}. Since Lemma~\ref{lemma:firstlayerp=3} shows that \(E(L)_\mathrm{tors}[3]=E(L_1)_\mathrm{tors}[3]\), the growth pattern for \(n\ge1\) is also given by Lemma~\ref{lemma:grothpatternp=ell=3,depth,non-cyclic}. It remains to determine the growth in the first layer.

    Suppose that \(E(K)[3^\infty]\cong \mathbb{Z}/3\mathbb{Z}\). We use the results of Theorem~\ref{theorem:ag-six-field} given by Ad\v{z}aga--Gu\v{z}vi\'{c}. Since \(L_1/\mathbb{Q}\) is a field extension of degree \(6\), \(E(L_1)[3^\infty]\) can only be one of the \(3\)-primary parts of those in Theorem~\ref{theorem:ag-six-field}, i.e., the trivial group, \(\mathbb{Z}/3\mathbb{Z}\), \(\mathbb{Z}/9\mathbb{Z}\), \(\mathbb{Z}/3\mathbb{Z}\times \mathbb{Z}/3\mathbb{Z}\), or \(\mathbb{Z}/9\mathbb{Z}\times \mathbb{Z}/3\mathbb{Z}\). Since \(E[3]\subseteq E(L_1)\), we now prove the claim. In the case \(E(K)[3^\infty]\cong \mathbb{Z}/9\mathbb{Z}\), the same argument applies.

    For the last claim, Lemma~\ref{lemma:firstlayerp=3} implies that \(E(L)[3]=E(L_1)[3]\). Since \(L\cap K_{cyc}=K_{cyc}^n\) for \(n\ge 1\), \(L_1=K_{cyc}^1\). This means, in this case, \(E[3] \subseteq E(L_1)\) but \(E[3] \nsubseteq E(K)\), that is, \(K(E[3])=K_{cyc}^1\), which contradicts Theorem~4.3 in \cite{xu2026torsiongrowthrationalelliptic}.

    Therefore, we are done.
\end{proof}

\section{The case \(p = 2\)}\label{section:p=2}

In this section we prove the \(p=2\) part of Theorem~\ref{theorem:main}.  Let \(K\) be a quadratic field and \(K_{\infty}\) the compositum of all \(\mathbb{Z}_2\)-extensions of \(K\). When \(K\) is imaginary quadratic, the eigenspace definition used for odd \(p\) is no longer available. We set \(K_{cyc}=K\mathbb{Q}_{\infty,2}\). Following \cite[Th\'eor\`eme~2.2, p.~14]{AIF_1983__33_4_1_0}, let \(K_{anti}\) denote the distinguished pro-dihedral \(\mathbb{Z}_2\)-extension of \(K\). Then
\[
K_\infty=K_{cyc}K_{anti},
\qquad
[K_{cyc}\cap K_{anti}:K]\in\{1,2\}.
\]

We follow the strategy used in the previous section. First we will deal with the roots of unity in \(K_\infty\).

\begin{lemma}\label{lem:no-cyclo-p>=5}
Let \(K\) be a quadratic field and \(L/K\) be a \(\mathbb{Z}_2\)-extension. 

\begin{itemize}
    \item For any prime \(p\ge 5\), \(\mu_{p^\infty}(L)\) is trivial for all \(K\) and all \(p\ge 5\).

    \item For \(p=3\), \(\zeta_9\notin L\). If \(\zeta_3\in L\), then either \(\zeta_3\in K\) or \(K(\zeta_3)=L_1\).

    \item For \(p=2\), if \(K\notin\{\mathbb{Q}(i),\mathbb{Q}(\sqrt{-2})\}\), then \(\zeta_8\notin L\). If \(i\in L\), then either \(i\in K\) or \(K(i)=L_1\).
\end{itemize}
\end{lemma}

\begin{proof}
\textbf{Case: \(p\ge 5 \)}. Since \(\mathbb{Q}(\zeta_p)/\mathbb{Q}\) is totally ramified at \(p\) with ramification index \(p-1\ge 4\), the extension \(K(\zeta_p)/K\) is ramified at a prime above \(p\). Since \(p\neq2\), this contradicts Proposition~\ref{prop:unramifiedoutsidep}. Hence \(\zeta_p\notin L\), and the first assertion follows.

\textbf{Case: \(p=3 \)}. Suppose that \(\zeta_9\in L\). Since every finite subextension of \(L/K\) has \(2\)-power degree, \([K(\zeta_9):K]\) must be a power of \(2\). On the other hand,
\[
[K(\zeta_9):K]=[\mathbb{Q}(\zeta_9):K\cap\mathbb{Q}(\zeta_9)]
\]
is either \(6\) or \(3\), because the only quadratic subfield of \(\mathbb{Q}(\zeta_9)\) is \(\mathbb{Q}(\sqrt{-3})\). This is impossible. If \(\zeta_3\in L\) but \(\zeta_3\notin K\), then \(K(\zeta_3)/K\) is quadratic and therefore equals the first layer \(L_1\).

\textbf{Case: \(p=2 \)}. Suppose that \(\zeta_8\in L\). If \(K\nsubseteq\mathbb{Q}(\zeta_8)\), then
\[
\Gal{K(\zeta_8)}{K}\cong\Gal{\mathbb{Q}(\zeta_8)}{\mathbb{Q}}\cong C_2\times C_2,
\]
which is impossible because every finite subextension of \(L/K\) is cyclic. Thus \(K\) is one of the three quadratic subfields \(\mathbb{Q}(i)\), \(\mathbb{Q}(\sqrt{-2})\), or \(\mathbb{Q}(\sqrt2)\). In the last case, Proposition~\ref{prop:allZpext} shows that \(K\) has a unique \(\mathbb{Z}_2\)-extension, namely the cyclotomic one, which is totally real. Hence \(\zeta_8\notin L\) when \(K\notin\{\mathbb{Q}(i),\mathbb{Q}(\sqrt{-2})\}\). Finally, if \(i\in L\) and \(i\notin K\), then \(K(i)/K\) is quadratic and therefore equals \(L_1\).
\end{proof}

\begin{remark}
    There are \(3\) quadratic subextensions in \(K_\infty/K\) for imaginary quadratic \(K\). One can show that the first layer of the cyclotomic extension is \(K(\sqrt{2})\). Then when \(K=\mathbb{Q}(\sqrt{-6})\), it is easy to see that \(K(\sqrt{2})\) contains \(\mu_3\) while \(\mu_3 \nsubseteq K\). One can also construct some examples where \(\mu_3 \nsubseteq K\) but \(\mu_3 \subset K'\) for a quadratic extension. The subtlety is that it is not straightforward to determine whether \(K'\) lies in \(K_\infty\). A similar situation occurs when considering \(\mu_4\). Also, the roots of unity in \(K_\infty\) are more complicated than in a single \(\mathbb{Z}_2\)-extension. In fact, one can show that \(\mu_{2^\infty}\subseteq K_\infty\) if and only if \(\zeta_4=i \in K_\infty\).
\end{remark}

\begin{theorem}\label{theorem:p=2}
Let \(E/\mathbb{Q}\) be an elliptic curve and \(K\) be a quadratic field. Let \(K_\infty\) be the compositum of all \(\mathbb{Z}_2\)-extensions of \(K\).
For every prime \(q>7\),
\[
E(K_{\infty})_{\mathrm{tors}}[q^{\infty}] = E(K)_{\mathrm{tors}}[q^{\infty}].
\]
\end{theorem}

\begin{proof}
Using the Weil pairing, Lemma~\ref{lem:no-cyclo-p>=5} implies that for any intermediate field  \(F\) of \(K_\infty/K\) with \(F/\mathbb{Q}\) Galois and any prime \(q\ge 5\), the group \(E(F)_{\mathrm{tors}}[q]\) is either trivial or cyclic of order \(q\).  Therefore, if
\[
E(F)_{\mathrm{tors}} \cong \mathbb{Z}/n\mathbb{Z} \times \mathbb{Z}/mn\mathbb{Z},
\]
Lemma~\ref{lemma:admits-n-isogeny} implies  that \(E\) admits an isogeny of degree \(m\), and by Theorem~\ref{theorem:isogeniesoverQ} the integer \(m\) belongs to the finite set
\[
\{1,\dots,19\}\cup\{21,25,27,37,43,67,163\}.
\]

Now let \(q\ge 5\) be a prime. The growth of the \(q^{\infty}\)-torsion in a \(\mathbb{Z}_2\)-extension is given by Lemma~\ref{lemma:growthpattern=ellnotp} with \(p=2\), \(q=q\). Exactly as in Section~\ref{section:p=3}, three possibilities occur. The only case that can lead to a strict increase is when \(E(K)_{\mathrm{tors}}[q^{\infty}]=0\) but \(E(F)_{\mathrm{tors}}[q^{\infty}]\neq 0\) for some finite Galois subextension \(F/\mathbb{Q}\) of \(K_\infty/\mathbb{Q}\). In that situation we have the following.

\begin{proposition}\label{prop:p=2-large-q-structure}
Assume \(q\ge 5\), \(E(K)_{\mathrm{tors}}[q^{\infty}]=0\), and \(E(F)_{\mathrm{tors}}[q^{\infty}]\neq 0\) for some finite Galois \(F/\mathbb{Q}\) inside \(K_\infty\). Then
\[
E(F)_{\mathrm{tors}}[q^{\infty}] \cong \mathbb{Z}/q^k\mathbb{Z},
\]
for some \(k\ge 1\) and \([\mathbb{Q}(P):\mathbb{Q}]=2^{s}\) for some \(s\ge 1\), where \(P\) is a generator of \(E(F)_{\mathrm{tors}}[q^{\infty}]\).
\end{proposition}
\begin{proof}
By Lemma~\ref{lemma:growthpattern=ellnotp}, the \(q^{\infty}\)-part is cyclic, say of order \(q^k\), and hence is \(G_{\mathbb{Q}}\)-stable. Let \(P\in E(F)\) be a point of order \(q^{k}\) (\(k\ge 1\)) with \(P\notin E(K)\). The Galois action on \(\langle P\rangle\) yields a character
\[
\psi:G_{\mathbb{Q}} \longrightarrow \Aut(\langle P\rangle) \cong (\mathbb{Z}/q^{k}\mathbb{Z})^{\times},
\]
so \([\mathbb{Q}(P):\mathbb{Q}]\) divides \(\varphi(q^{k})=(q-1)q^{k-1}\). Because \(\mathbb{Q}(P)\subseteq K_{\infty}\) and \(K_{\infty}/K\) is a pro-\(2\) extension, \([\mathbb{Q}(P):\mathbb{Q}]\) is a power of \(2\), say \(2^{s}\).
\end{proof}

We now exclude primes for which such a point \(P\) would contradict known results on torsion over number fields of low degree.

\begin{proposition}\label{prop:p=2-exclude}
Under the hypotheses of Proposition~\ref{prop:p=2-large-q-structure}, the prime \(q\) cannot belong to the set \(\{11,13,17,19,37,43,67,163\}\).
\end{proposition}

\begin{proof}
\emph{Case \(q=11, 19,43,67,163\).} Here \(q-1\) is even but not divisible by \(4\). From \(2^{s}\mid(q-1)\) we obtain \(s=1\). Hence \(\mathbb{Q}(P)\) is a quadratic field. Theorem~\ref{theorem:NajmanK} shows that no elliptic curve \(E/\mathbb{Q}\) acquires a point of order \(q\) over any quadratic field, contradiction.

\emph{Case \(q=37\).} We have \(37-1=36\), so \(s\) can be \(1\) or \(2\). The quadratic case is ruled out exactly as above. If \(s=2\), then \([\mathbb{Q}(P):\mathbb{Q}]=4\), and \(F\) contains a quartic Galois extension of \(\mathbb{Q}\). However, Theorem~\ref{theorem:chou-quartic-galois} states that an elliptic curve \(E/\mathbb{Q}\) cannot have a \(37\)-torsion point over any quartic Galois extension of \(\mathbb{Q}\). Hence this case is also impossible.

\emph{Cases \(q=13\) and \(q=17\).} Replacing the generator in Proposition~\ref{prop:p=2-large-q-structure} by a suitable multiple, let \(P\) have order \(q\). The action on \(\langle P\rangle\) gives a Dirichlet character
\[
\psi:G_{\mathbb{Q}}\longrightarrow\Aut(\langle P\rangle)\cong(\mathbb{Z}/q\mathbb{Z})^\times
\]
of \(2\)-power order. Let \(N\) be its conductor. For every odd prime \(\ell\mid N\), the image of inertia at \(\ell\) has order at most \(2\). Indeed, \(K\mathbb{Q}(P)\subseteq F\), while \(F/K\) is unramified outside \(2\), so the odd inertia can only come from the quadratic extension \(K/\mathbb{Q}\).

Write \(\psi\) as the product of its local components. Since \((\mathbb{Z}/q\mathbb{Z})^\times\) is cyclic and has a unique element of order \(2\), the product \(\psi'\) of the components at the odd primes is a quadratic or trivial  Dirichlet character. Hence \(\psi_1=\psi'\psi\) has conductor a power of \(2\). The character \(\psi'\) corresponds to a rational quadratic twist, and the image of \(P\) on this twist is acted on through \(\psi_1\).

By Kronecker--Weber, the field cut out by \(\psi_1\) is contained in
\[
\mathbb{Q}(\mu_{2^\infty})=\mathbb{Q}_{\infty,2}(i).
\]
The finite factor of \(\Gal{\mathbb{Q}(\mu_{2^\infty})}{\mathbb{Q}}\cong C_2\times\mathbb{Z}_2\) acts through either the trivial character or the unique quadratic character. In the latter case, twist once more by \(\chi_{-4}\). Thus, for some rational quadratic twist \(E^d\), the group \(E^d(\mathbb{Q}_{\infty,2})\) contains a point of order \(q\). This contradicts Theorem~\ref{theorem:chouQinfty}(ii), whose list contains neither \(13\)- nor \(17\)-torsion.
\end{proof}

Now let \(q>7\) be any prime. Proposition~\ref{prop:p=2-large-q-structure} and Proposition~\ref{prop:p=2-exclude} show that such a \(q\)-torsion point cannot appear in \(E(F)\) for any finite Galois extension \(F/\mathbb{Q}\) in \(K_\infty/K\). Therefore,
\[
E(K_{\infty})_{\mathrm{tors}}[q^{\infty}] = E(K)_{\mathrm{tors}}[q^{\infty}]
\]
for every \(q>7\). 
\end{proof}

\begin{remark}
Lemma~\ref{lemma:ell-twist} is not applicable in this section, as it requires the prime \(p\) to be odd. However, Xu--Yuan \cite[Lemma~5.1]{xu2026torsiongrowthrationalelliptic} generalize this lemma and the methods in this theorem  to \(p=2\) with any \(q\ge 5\) and give an elegant proof of this theorem using Theorem~\ref{theorem:chouQinfty}. 
\end{remark}

\begin{corollary}
    Let \(E/\mathbb{Q}\) be an elliptic curve and \(K\) be a quadratic field. Let \(K_\infty\) be the compositum of all \(\mathbb{Z}_2\)-extensions of \(K\). Then
    \[
    E(K_\infty)_{\mathrm{tors}}[q^\infty]=0
    \]
    for every prime \(q>7\).
\end{corollary}

\begin{proof}
    This is a direct combination of Theorem~\ref{theorem:p=2} and Theorem~\ref{theorem:NajmanK}.
\end{proof}

For \(q=7\), Theorem~\ref{theorem:isogeniesoverQ} shows that there is no rational \(49\)-isogeny. Hence Proposition~\ref{prop:p=2-large-q-structure} implies that \(E(K_\infty)_{\mathrm{tors}}[7^\infty]\) has order at most \(7\).

\begin{proposition}\label{prop:p=2,q=7}
   Let \(E/\mathbb{Q}\) be an elliptic curve and \(K\) be a quadratic field. Let \(K_\infty\) be the compositum of all \(\mathbb{Z}_2\)-extensions of \(K\). Then
   \[
   E(K_\infty)_{\mathrm{tors}}[7^\infty]=E(K_\infty^{(1)})_{\mathrm{tors}}[7^\infty].
   \]
   Moreover, \(E(K_\infty)_{\mathrm{tors}}[7^\infty]\) is either trivial or isomorphic to \(\mathbb{Z}/7\mathbb{Z}\).
\end{proposition}

\begin{proof}
    By the Weil pairing and Lemma~\ref{lem:no-cyclo-p>=5}, the \(7\)-primary part is cyclic of order at most \(7\). If \(P\) is a point of order \(7\), then \([K(P):K]\) is a power of \(2\) dividing \(|\Aut(\langle P\rangle)|=6\). Hence \([K(P):K]\le2\), so \(P\) is defined over \(K_\infty^{(1)}\).
\end{proof}

\begin{proposition}\label{prop:p=2,q=5}
   Let \(E/\mathbb{Q}\) be an elliptic curve and \(K\) be a quadratic field. Let \(K_\infty\) be the compositum of all \(\mathbb{Z}_2\)-extensions of \(K\). Then
   \[
   E(K_\infty)_{\mathrm{tors}}[5^\infty]=E(K_\infty^{(2)})_{\mathrm{tors}}[5^\infty].
   \]
   Moreover, \(E(K_\infty)_{\mathrm{tors}}[5^\infty]\) is either trivial or isomorphic to \(\mathbb{Z}/5\mathbb{Z}\).
\end{proposition}

\begin{proof}
    By the Weil pairing and Lemma~\ref{lem:no-cyclo-p>=5}, the \(5\)-primary part is cyclic. For a point \(P\) of order \(5^k\), the degree \([K(P):K]\) is a power of \(2\) dividing \(\varphi(5^k)=4\cdot5^{k-1}\), and hence it divides \(4\). Thus every such point is defined over \(K_\infty^{(2)}\).

    The last claim is not immediate because \(E\) can admit a rational \(25\)-isogeny. If a point of order \(25\) existed, \cite[Lemma~5.1]{xu2026torsiongrowthrationalelliptic} would give a rational quadratic twist \(E^d\) having a point of order \(25\) over \(\mathbb{Q}_{\infty,2}\), contradicting Theorem~\ref{theorem:chouQinfty}(ii).
\end{proof}

\begin{proposition}\label{prop:p=2,q=3}
Let \(E/\mathbb{Q}\) be an elliptic curve, let \(K\) be a quadratic field, and let \(L/K\) be a \(\mathbb{Z}_2\)-extension with layers \(L_n\). Let \(
M_n=E(L_n)_{\mathrm{tors}}[3^\infty]\), and  \(M=E(L)_{\mathrm{tors}}[3^\infty].\)
Then
\[
M\in\left\{0,\ \mathbb{Z}/3\mathbb{Z},\ \mathbb{Z}/9\mathbb{Z},\ (\mathbb{Z}/3\mathbb{Z})^2\right\}
\quad\text{and}\quad M=M_3.
\]
Moreover, the following hold.
\begin{itemize}
    \item If \(\mu_3\nsubseteq L\), then \(M=M_1\). If \(E(K)[3]\neq0\), then \(M=M_0\). If \(E(K)[3]=0\), then \(M\) is trivial, \(\mathbb{Z}/3\mathbb{Z}\), or \(\mathbb{Z}/9\mathbb{Z}\).

    \item If \(\mu_3\subseteq L\), then at most one of the inclusions
    \[
    M_0\subseteq M_1\subseteq M_2\subseteq M_3
    \]
    is strict. If \(E(K)[3]\neq0\), then either there is no growth, or
    \[
    M_0\cong\mathbb{Z}/3\mathbb{Z},\qquad
    M_1=M\cong(\mathbb{Z}/3\mathbb{Z})^2.
    \]
    If \(E(K)[3]=0\), then there may be no growth, the possible non-trivial growth patterns are
    \begin{itemize}
        \item growth at \(L_1\), with \(M=M_1\) isomorphic to \(\mathbb{Z}/3\mathbb{Z}\), \(\mathbb{Z}/9\mathbb{Z}\), or \((\mathbb{Z}/3\mathbb{Z})^2\);
        \item growth at \(L_2\) or \(L_3\), with \(M\cong(\mathbb{Z}/3\mathbb{Z})^2\).
    \end{itemize}
\end{itemize}
Finally, if \(K_\infty\) denotes the compositum of all \(\mathbb{Z}_2\)-extensions of \(K\), then
\[
E(K_\infty)_{\mathrm{tors}}[3^\infty]
=E(K_\infty^{(3)})_{\mathrm{tors}}[3^\infty].
\]
\end{proposition}

\begin{proof}
By \cite[Theorem~5.4]{xu2026torsiongrowthrationalelliptic},
\[
M\in\left\{0,\ \mathbb{Z}/3\mathbb{Z},\ \mathbb{Z}/9\mathbb{Z},\ (\mathbb{Z}/3\mathbb{Z})^2\right\}. \tag{\(\ast\)}
\]
For \(q=3\) and \(p=2\), the integers in Lemma~\ref{lemma:structure} satisfy
\(d_1=1, d_2=d_3=2, d_j>2 \quad (j\ge4).\)
Thus no growth can occur after \(L_3\).

A new summand appearing at \(L_2\) or \(L_3\) requires two generators, whereas a new summand at \(L_1\) requires at least one generator. Since every group in \((\ast)\) has at most two generators, growth can occur at most once. More precisely, after a non-trivial cyclic summand has appeared at \(L_1\), a further new summand at \(L_2\) or \(L_3\) would require at least three generators in total. If two generators appear at any layer, there is no room for further growth.

Suppose first that \(\mu_3\nsubseteq L\). The Weil pairing shows that \(M\) is cyclic. Hence a non-trivial new summand cannot first appear at \(L_2\) or \(L_3\), and so \(M=M_1\). If \(M_0\neq0\), Lemma~\ref{lemma:growthpattern=ellnotp} and \((\ast)\) leave no possible strict inclusion. If \(M_0=0\), the cyclic groups in \((\ast)\) give the asserted alternatives.

Now suppose that \(\mu_3\subseteq L\). If \(M_0\neq0\), growth at \(L_2\) or \(L_3\) would give at least three generators. By Lemma~\ref{lemma:growthpattern=ellnotp} and \((\ast)\), the only possible strict inclusion at \(L_1\) is
\[
\mathbb{Z}/3\mathbb{Z}\subset(\mathbb{Z}/3\mathbb{Z})^2.
\]
If \(M_0=0\), the number of generators required at each layer, together with \((\ast)\), gives precisely the listed possible non-trivial patterns. 

It remains to prove the assertion for \(K_\infty\). The same theorem of Xu--Yuan gives
\[
E(K_\infty)_{\mathrm{tors}}[3^\infty]
\in\left\{0,\ \mathbb{Z}/3\mathbb{Z},\ \mathbb{Z}/9\mathbb{Z},\ (\mathbb{Z}/3\mathbb{Z})^2\right\}.
\]
The exponent of a Sylow \(2\)-subgroup of the automorphism group of each group in this list divides \(8\). Therefore \(8\Gal{K_\infty}{K}\) acts trivially, and its fixed field is \(K_\infty^{(3)}\). This proves the final equality.
\end{proof}

\begin{remark}
    It remains to provide some concrete examples demonstrating that these growth phenomena actually occur. This is a separate issue, which we intend to investigate in future work.
\end{remark}

Now we turn to the case where \(p=q=2\).

\begin{lemma}\label{lemma:firstlayerp=2}
    Let \(E/\mathbb{Q}\) be an elliptic curve and let \(K\) be a number field. For any \(\mathbb{Z}_2\)-extension \(L/K\), we have
    \[
    E(L)_{\mathrm{tors}}[2^n]=E(L_n)_\mathrm{tors}[2^n],
    \]
    where \(L_n\) is the \(n\)-th layer of \(L/K\), for \(n\ge1\).
\end{lemma}

\begin{proof}
Let \(P\in E(L)[2^n]\), and let \(\gamma\) be a topological generator of \(\Gal{L}{K}\). Choose an extension \(\widetilde\gamma\in G_K\) and let \(A\) be its image in \(\GL_2(\mathbb{Z}/2^n\mathbb{Z})\). The degree \([K(P):K]\) is the length of the orbit of \(P\) under \(\gamma\). It is a power of \(2\) dividing the order of \(A\).

The \(2\)-part of the order of the reduction of \(A\) in \(\GL_2(\mathbb{F}_2)\cong S_3\) divides \(2\). Moreover, the kernel of
\[
\GL_2(\mathbb{Z}/2^n\mathbb{Z})\longrightarrow\GL_2(\mathbb{F}_2)
\]
has exponent dividing \(2^{n-1}\), by iterating the binomial congruence \((I+2B)^{2^r}\equiv I\pmod{2^{r+1}}\). Hence the \(2\)-part of the order of \(A\) divides \(2^n\). Therefore \([K(P):K]\mid2^n\), and the uniqueness of the layers of \(L/K\) gives \(K(P)\subseteq L_n\).
\end{proof}

\begin{remark}
    Lemmas~\ref{lemma:firstlayerp=3} and~\ref{lemma:firstlayerp=2} can be refined by taking into account the structure of \(E(L)_{\mathrm{tors}}[p^\infty]\) for \(p=2,3\). In this article, we mainly use \(n=1\).
\end{remark}

\begin{proposition}\label{prop:p=2-infinite-dichotomy}
    Let \(E/\mathbb{Q}\) be an elliptic curve and let \(K\) be a quadratic field, with \(K_\infty\) the compositum of all \(\mathbb{Z}_2\)-extensions of \(K\). Then either \(E(K_\infty)_\mathrm{tors}[2^\infty]\) is finite or \(E[2^\infty]\subseteq E(K_\infty)\).
\end{proposition}

\begin{proof}
Suppose that \(E(K_\infty)_{\mathrm{tors}}[2^\infty]\) is infinite. It then contains points \(P_n\) of order \(2^n\) for arbitrarily large \(n\). Let \(F_n\) be the Galois closure of \(K(P_n)\) over \(\mathbb{Q}\). Since \(K_\infty/\mathbb{Q}\) is Galois, \(F_n\subseteq K_\infty\). Write
\[
E(F_n)_{\mathrm{tors}}[2^\infty]\cong
\mathbb{Z}/2^{a_n}\mathbb{Z}\times\mathbb{Z}/2^{b_n}\mathbb{Z},
\qquad a_n\ge b_n\ge0.
\]
Lemma~\ref{lemma:admits-n-isogeny} and Theorem~\ref{theorem:isogeniesoverQ} give \(a_n-b_n\le4\). Since \(a_n\ge n\), it follows that \(b_n\ge n-4\). Thus, for every \(m\ge1\), choosing \(n\ge m+4\) gives
\[
E[2^m]\subseteq E(F_n)\subseteq E(K_\infty).
\]
Hence \(E[2^\infty]\subseteq E(K_\infty)\).
\end{proof}

\begin{theorem}\label{theorem:CMclass,p=2}
    Let \(E/\mathbb{Q}\) be an elliptic curve and \(K\) be a quadratic field, with \(K_\infty\) the compositum of all \(\mathbb{Z}_2\)-extensions of \(K\). Then \(E[2^\infty]\subseteq E(K_\infty)\) if and only if \(E\) satisfies the following conditions:
    \begin{itemize}
        \item \(E\) has CM by an order in \(K\), and necessarily \(K=\mathbb{Q}(i)\) or \(\mathbb{Q}(\sqrt{-2})\);

        \item \(E\) has good reduction outside \(2\).
    \end{itemize}

    Moreover, all elliptic curves satisfying these conditions are isomorphic to one of the following curves, given by their Cremona label:
    \begin{itemize}
        \item CM by \(\mathbb{Z}[i]\):  \(\mathrm{32a1}, \mathrm{32a2}, \mathrm{64a1}, \mathrm{64a4}, \mathrm{256b1}, \mathrm{256b2}, \mathrm{256c1}, \mathrm{256c2}\).

        \item CM by \(\mathbb{Z}[2i]\): \(\mathrm{32a3}, \mathrm{32a4}, \mathrm{64a2}, \mathrm{64a3}\).

        \item CM by \(\mathbb{Z}[\sqrt{-2}]\): \(\mathrm{256a1}, \mathrm{256a2}, \mathrm{256d1}, \mathrm{256d2}\).
    \end{itemize}
    These curves are exactly those listed in \cite[Table~3]{xu2026torsiongrowthrationalelliptic}.
\end{theorem}

\begin{proof}
Suppose first that \(E[2^\infty]\subseteq E(K_\infty)\). If \(K\) is real, then \(K_\infty\) is its cyclotomic \(\mathbb{Z}_2\)-extension and is totally real. The Weil pairing would give \(\mu_{2^\infty}\subseteq K_\infty\), a contradiction. Hence \(K\) is imaginary.

Now consider the Galois representation. The representation on \(E[2^\infty]\) factors through the abelian group \(\Gal{K_\infty}{K}\):
\[
\rho:\Gal{K_\infty}{K}\longrightarrow\GL(E[2^\infty])\cong\GL_2(\mathbb{Z}_2).
\]
If \(E\) were non-CM, Serre's open image theorem~\cite[Th\'{e}or\`eme~3, p.~299]{Serre1972} would imply that the image of \(G_K\) is open in \(\GL_2(\mathbb{Z}_2)\), and in particular is not abelian. Thus \(E\) has CM. The same argument in Theorem~\ref{theorem:CMclass,p=3} shows  its CM field is \(K\). Therefore  \(E\) has CM by an order in \(K\).

Every \(\mathbb{Z}_2\)-extension is unramified outside \(2\). Thus \(K(E[2^\infty])/K\), being contained in \(K_\infty/K\), is unramified outside \(2\). The N\'{e}ron--Ogg--Shafarevich criterion now shows that \(E/K\) has good reduction outside \(2\).
    
We next show that \(K=\mathbb{Q}(i)\) or \(\mathbb{Q}(\sqrt{-2})\). Since \(E/\mathbb{Q}\) has CM by an order \(\mathcal{O}\), the discriminant of its CM order belongs to
    \[
    D_\mathcal{O}\in \{-3,-4,-7,-8,-11,-12,-16,-19,-27,-28,-43,-67,-163\}.
    \]
By Cremona--Pacetti~\cite{CremonaPacetti2019}, if an odd prime \(q\) divides the discriminant of the CM field, then the curve does not acquire good reduction at \(q\) over an abelian extension of \(\mathbb{Q}_q\). Since \(K/\mathbb{Q}\) is abelian and \(E/K\) has good reduction outside \(2\), no such odd prime can occur. The displayed list therefore leaves only the CM fields \(\mathbb{Q}(i)\) and \(\mathbb{Q}(\sqrt{-2})\).

It remains to descend good reduction from \(K\) to \(\mathbb{Q}\). Only \(2\) ramifies in either of these two fields. If an odd prime \(q\) splits in \(K\), then \(K_{\mathfrak q}=\mathbb{Q}_q\). If \(q\) is inert, then \(K_{\mathfrak q}/\mathbb{Q}_q\) is unramified and the inertia groups coincide. In both cases the N\'{e}ron--Ogg--Shafarevich criterion shows that good reduction over \(K\) above \(q\) implies good reduction over \(\mathbb{Q}\) at \(q\). Hence \(E/\mathbb{Q}\) has good reduction outside \(2\).

Conversely, suppose that \(E\) has CM by an order in \(K=\mathbb{Q}(i)\) or \(\mathbb{Q}(\sqrt{-2})\), and that \(E/\mathbb{Q}\) has good reduction outside \(2\). Then \(E/K\) has good reduction outside the prime above \(2\), so \cite[Corollary~4.2]{xu2026torsiongrowthrationalelliptic} gives \(E[2^\infty]\subseteq E(K_\infty)\).

    For the classification, let \(\mathcal{O}=\mathbb{Z}+f\mathcal{O}_K\). We discuss case by case.

    \textbf{Case: \(K=\mathbb{Q}(i)\)}. Then \(f=1,2\) since \(E\) is rational. For \(f=1\), the discriminant is \(-4\) with \(j\)-invariant \(1728\). Then SageMath gives eight curves.
    \[\mathrm{32a1}, \mathrm{32a2}, \mathrm{64a1}, \mathrm{64a4}, \mathrm{256b1}, \mathrm{256b2}, \mathrm{256c1}, \mathrm{256c2}.\]
    For \(f=2\), the discriminant is \(-16\) with \(j\)-invariant \(287496\). Then SageMath gives four curves:
    \[\mathrm{32a3}, \mathrm{32a4}, \mathrm{64a2}, \mathrm{64a3}.\]

    \textbf{Case:\(K=\mathbb{Q}(\sqrt{-2})\)}. Then \(f=1\), since \(E\) is rational.
    In this case, the discriminant is \(-8\) with \(j\)-invariant \(8000\). Then SageMath gives four curves.
    \[\mathrm{256a1}, \mathrm{256a2}, \mathrm{256d1}, \mathrm{256d2}.\]

    They are all listed in \cite[Table~3]{xu2026torsiongrowthrationalelliptic}.
\end{proof}

\begin{theorem}\label{theorem:p=2singleZpfinite}
     Let \(E/\mathbb{Q}\) be an elliptic curve and \(K\) be a quadratic field, with \(K_\infty\) the compositum of all \(\mathbb{Z}_2\)-extensions of \(K\). Let \(L/K\) be a \(\mathbb{Z}_2\)-extension. Then \(E(L)_\mathrm{tors}[2^\infty]\) is finite.
\end{theorem}

\begin{proof}
If \(E(K_\infty)_{\mathrm{tors}}[2^\infty]\) is finite, there is nothing to prove. Suppose that it is infinite. Proposition~\ref{prop:p=2-infinite-dichotomy} then gives \(E[2^\infty]\subseteq E(K_\infty)\), and Theorem~\ref{theorem:CMclass,p=2} shows that \(K=\mathbb{Q}(i)\) or \(\mathbb{Q}(\sqrt{-2})\) and that \(E\) has CM by an order \(\mathcal{O}\) in \(K\).

By CM theory the representation factors as
\[
\rho:\Gal{K_\infty}{K}\longrightarrow(\mathcal{O}\otimes\mathbb{Z}_2)^\times.
\]
Its image has \(\mathbb{Z}_2\)-rank two. Hence its kernel has rank zero,  since \(\Gal{K_\infty}{K}\) is torsion-free, \(\rho\) is injective. Write
\[
\Gal{K_\infty}{L}=\overline{\langle\gamma\rangle}.
\]
Then \(\rho(\gamma)\neq1\). Since \(2\) is ramified in \(K\), the algebra \(K\otimes\mathbb{Q}_2\) is a field, and multiplication by the non-zero element \(\rho(\gamma)-1\) is an automorphism of \(V_2(E)=T_2(E)\otimes\mathbb{Q}_2\). It follows that
\[
\coker\bigl(\rho(\gamma)-1:T_2(E)\longrightarrow T_2(E)\bigr)
\]
is finite. Equivalently, the kernel of \(\rho(\gamma)-1\) on \(E[2^\infty]\) is finite. This kernel is \(E(L)_{\mathrm{tors}}[2^\infty]\), which proves the theorem.
\end{proof}

\begin{lemma}\label{lemma:grothpatternp=ell=2,depth,non-cyclic}
    Let \(E/\mathbb{Q}\) be an elliptic curve, let \(K\) be a quadratic field, and let \(L/K\) be a \(\mathbb{Z}_2\)-extension with intermediate subfields \(L_n\) for \(n\ge 1\), and set \(L_0=K\). Suppose that \(E(L)_{\mathrm{tors}}[2]\cong\mathbb{Z}/2\mathbb{Z}\times\mathbb{Z}/2\mathbb{Z}\). For \(n\ge0\), let
    \[
     G_n:=\frac{E(L_{n+1})_{\mathrm{tors}}[2^\infty]}{ E(L_{n})_{\mathrm{tors}}[2^\infty]}.
    \]
     Then there exists a nonnegative integer \(s\) such that \(G_n\neq1\) for \(n<s\) and \(G_n=1\) for \(n\ge s\). In other words, growth occurs precisely in the first \(s\) layers.
\end{lemma}

\begin{proof}
Let  \(M=E(L)_{\mathrm{tors}}[2^\infty]\). By Theorem~\ref{theorem:p=2singleZpfinite}, \(M\) is finite. Let \(\gamma\) be a topological generator of \(\Gamma=\Gal{L}{K}\), and suppose that the action of \(\gamma\) on \(M\) has order \(2^r\). Then
\[
E(L_n)_{\mathrm{tors}}[2^\infty]=M^{\gamma^{2^n}}.
\]
If \(n<r\), the automorphism \(g=\gamma^{2^n}\) is non-trivial on \(M\), and Lemma~\ref{lemma:fixed-point-growth} gives
\[
M^{\gamma^{2^n}}\subsetneq M^{\gamma^{2^{n+1}}}.
\]
If \(n\ge r\), both fixed subgroups equal \(M\). Thus the conclusion holds with \(s=r\).
\end{proof}

\begin{remark}
    When \(p=2\), \(2G_n\) is not trivial in general. For example, let \(E\) be the elliptic curve with Cremona label \(\mathrm{24a6}\) and \(K=\mathbb{Q}(\sqrt{-2})\). Then \(E(K)[2^\infty]\cong \mathbb{Z}/2\mathbb{Z}\times \mathbb{Z}/2\mathbb{Z}\) and \(E(K(\sqrt{2}))_\mathrm{tors}[2^\infty]\cong \mathbb{Z}/8\mathbb{Z} \times \mathbb{Z}/2\mathbb{Z}\) in LMFDB. Note that \(K(\sqrt{2})\) is the first layer of the cyclotomic \(\mathbb{Z}_2\)-extension. Hence \(G_0\cong \mathbb{Z}/4\mathbb{Z}\) even if \(E[2]\subseteq E(L)\).
\end{remark}

\begin{theorem}\label{theorem:growth-pattern,2}
    Let \(E/\mathbb{Q}\) be an elliptic curve and \(K\) be a quadratic field. Let \(L/K\) be a \(\mathbb{Z}_2\)-extension with intermediate fields \(L_n\). Set \(L_0=K\). Then the growth pattern is one of the following:

    \begin{itemize}
        \item \(E(L)_{\mathrm{tors}}[2]=0\) if and only if  \(E(K)_{\mathrm{tors}}[2]=0\).
    
        \item If \(E(L)_{\mathrm{tors}}[2]\cong \mathbb{Z}/2\mathbb{Z}\), \(E(L)_{\mathrm{tors}}[2^\infty]\cong \mathbb{Z}/2^s\mathbb{Z}\), and \(E(K)_{\mathrm{tors}}[2^\infty]\cong \mathbb{Z}/2^k\mathbb{Z}\), then:
            \begin{itemize}
                \item If \(k\ge 2\), then \(E(L_n)_{\mathrm{tors}}[2^\infty]\cong \mathbb{Z}/2^{\min \{k+n,s\} } \mathbb{Z}\) for any \(n\ge 0\).
                \item If \(k=1\) and \(E(L_1)_{\mathrm{tors}}[2^\infty]\cong \mathbb{Z}/2^{r_1} \mathbb{Z}\), then \(E(L_n)_{\mathrm{tors}}[2^\infty]\cong \mathbb{Z}/2^{\min \{r_1+n-1,s\} } \mathbb{Z}\) for any \(n\ge 1\).
            \end{itemize}
        \item If \(E(L)_{\mathrm{tors}}[2]\cong \mathbb{Z}/2\mathbb{Z}\times\mathbb{Z}/2\mathbb{Z}\), let
        \[
        G_n=\frac{E(L_{n+1})_{\mathrm{tors}}[2^\infty]}{E(L_n)_{\mathrm{tors}}[2^\infty]}.
        \]
        Then there exists a nonnegative integer \(s\) such that \(G_n\neq1\) for \(n<s\) and \(G_n=1\) for \(n\ge s\). In other words, growth occurs precisely in the first \(s\) layers.
    \end{itemize}
\end{theorem}

\begin{proof}
    Denote by \(r\) the \(\mathbb{Z}/2\mathbb{Z}\)-rank of \(E(L)_{\mathrm{tors}}[2]\). Then \(r=0\) is Lemma~\ref{lemma:grothpatternp=ell}, \(r=1\) is Lemma~\ref{lemma:grothpatternp=ell,depth}, and \(r=2\) is Lemma~\ref{lemma:grothpatternp=ell=2,depth,non-cyclic}.
\end{proof}

\begin{remark}
    Compared with \(p=3\), Theorem~\ref{theorem:growth-pattern,2} is less explicit than Theorem~\ref{theorem:growthppterQneqQ-3} and Theorem~\ref{theorem:growth-patternQ(-3),-3}.
\end{remark}

\bibliographystyle{alpha}
\bibliography{ref}

\end{document}